\documentclass[reqno]{amsart}
\usepackage[dvipsnames]{xcolor}
\usepackage{amssymb,amscd,bbm,cite,verbatim,graphicx,mathtools}
\usepackage[mathscr]{euscript}
\usepackage[colorlinks]{hyperref}
\newif\ifOld\Oldfalse %\Oldtrue ignores all changes
\newif\ifChanges %\Changestrue print additions in green and deletions in red. \Changesfalse prints additions in default and ignores deletions
\Changestrue
\newtheorem{thm}{Theorem}[section]
\newtheorem{cor}[thm]{Corollary}
\newtheorem{lem}[thm]{Lemma}

\theoremstyle{definition}
\newtheorem{defn}[thm]{Definition}
\newtheorem{remark}[thm]{Remark}
\newtheorem{hyp}[thm]{Hypothesis}

\numberwithin{equation}{section}

\newcommand{\ie}{i.e.}

\newcommand{\vertiii}[1]{{\left\vert\kern-0.25ex\left\vert\kern-0.25ex\left\vert #1\right\vert\kern-0.25ex\right\vert\kern-0.25ex\right\vert}}
\newcommand{\ov}{\overline}

\renewcommand{\Im}{\operatorname{Im}}
\newcommand{\supp}{\operatorname{supp}}
\renewcommand{\span}{\operatorname{span}}
\newcommand{\e}{\operatorname{e}}

\newcommand{\dom}{\operatorname{dom}}
\newcommand{\ran}{\operatorname{ran}}
\newcommand{\tr}{\operatorname{tr}}
\newcommand{\BV}{\operatorname{BV}}
\newcommand{\BVl}{\operatorname{BV}_{\rm loc}}
\newcommand{\Var}{\operatorname{Var}}
\newcommand{\sgn}{\operatorname{sgn}}
\newcommand{\sm}[1]{\big(\begin{smallmatrix}#1\end{smallmatrix}\big)}
\newcommand{\bb}[1]{{\mathbb{#1}}}
\newcommand{\mc}[1]{{\mathcal{#1}}}

\newcommand{\id}{\mathbbm 1}
\newcommand{\cbv}{e}

\newcommand{\ol}{{\overline{\lambda}}}
\newcommand{\la}{\lambda}

\newcommand{\loc}{{\rm loc}}
\newcommand{\mx}{{\rm max}}
\newcommand{\mn}{{\rm min}}
\newcommand{\iOmega}{(a,b)}
\newcommand{\pim}{\pi}
\newcommand{\Pim}{\Pi}

\newcommand{\RN}{{\tilde W}}
\newcommand{\exsum}{\times}

\newcommand{\<}{\langle}
\renewcommand{\>}{\rangle}

\begin{document}
\title[Spectral Theory]{Spectral Theory for Systems of Ordinary Differential Equations with Distributional Coefficients}%
\author{Ahmed Ghatasheh and Rudi Weikard}%
\address{A.G.: Mathematics Department, The Ohio State University at Marion, Marion, OH 43302, USA}%
\email{ghatasheh.1@osu.edu}%
\address{R.W.: Department of Mathematics, University of Alabama at Birmingham, Birmingham, AL 35226-1170, USA}%
\email{weikard@uab.edu}%

\date{18. September 2019}%
\thanks{\copyright 2019. This manuscript version is made available under the CC-BY-NC-ND 4.0 license \texttt{http://creativecommons.org/licenses/by-nc-nd/4.0/}}%

%\keywords{47B25, 47A10, 42A38, 47A06;  also 34L05, 81Q10}

\begin{abstract}
We study the spectral theory for the first-order system $Ju'+qu=wf$ of differential equations on the real interval $(a,b)$ when $J$ is a constant, invertible skew-Hermitian matrix and $q$ and $w$ are matrices whose entries are distributions of order zero with $q$ Hermitian and $w$ non-negative.
Also, we do not pose the definiteness condition often required for the coefficients of the equation.
Specifically, we construct minimal and maximal relations, and study self-adjoint restrictions of the maximal relation.
For these we determine Green's function and prove the existence of a spectral (or generalized Fourier) transformation.
We have a closer look at the special cases when the endpoints of the interval $(a,b)$ are regular as well as the case of a $2\times2$ system.
Two appendices provide necessary details on distributions of order zero and the abstract spectral theory for relations.
\end{abstract}

\maketitle

\section{Introduction}
In this paper we study the spectral theory for the first-order system
$$Ju'+qu=wf$$
of differential equations on the real interval $\iOmega$ when $J$ is a constant, invertible skew-Hermitian matrix and $q$ and $w$ are matrices whose entries are distributions of order zero\footnote{Appendix \ref{ADM} gathers the basic properties of distributions of order $0$ and the closely related functions of locally bounded variation.
In particular, we recall that on compact subintervals of $(a,b)$ distributions of order zero may be thought of as measures and we may use both words interchangeably.}
with $q$ Hermitian and $w$ non-negative.
For appropriate functions $f$ (those for which the components of $wf$ are also distributions of order zero) solutions of this differential equation will have to be sought among the functions of locally bounded variation.
To pursue our program we will have to utilize a spectral theory based on linear relations rather than linear operators.\footnote{A review of the relevant material is given in Appendix \ref{ALR}.}
Our main result is Theorem~\ref{t:main}.

Interest in the spectral theory for differential equations arose in Fourier's work \cite{Fourier1822} on the heat equation.
In the 1830s Sturm and Liouville generalized Fouri\-er's ideas to cover what is now known as the Sturm-Liouville equation, viz., the equation $-(pu')'+qu=\la w u$, posed on a finite interval.
In 1923 Birkhoff and Langer \cite{zbMATH02599961} treated a first-order system of differential equations extending earlier work by Birkhoff \cite{MR1500818} on a scalar equation of higher order (with continuous coefficients).
In 1910 Weyl \cite{41.0343.01} showed how to deal with singular problems, i.e., problems on unbounded intervals or problems where the coefficients are allowed to have singularities at the endpoints.
Another generalization, albeit sometimes ignored, was to admit locally integrable rather than continuous coefficients\footnote{In his book \cite{MR2170950} Zettl writes: ``It is surprising how many authors ... assume continuity of the coefficients when local Lebesgue integrability suffices.''}.
It is, however, desirable, to allow coefficients even more general than locally integrable ones.
The first to do so (as far as we know) was Krein \cite{MR0054078} in 1952 when he modeled a vibrating string by a Sturm-Liouville equation with $p=1$ and $q=0$ but $w$ a positive Lebesgue-Stieltjes measure.
Shortly thereafter Kac \cite{MR0080835} generalized this approach by also allowing $q$ to be a measure, see also Mingarelli \cite{MR706255}.
Other contributions were made by Feller \cite{MR0068082} in 1955 who introduced measures implicitly through Radon-Nikodym derivatives and by Gesztesy and Holden \cite{MR914699} in 1987 who described Schr\"odinger equations with point interactions, specifically $\delta'$-interactions, see also Albeverio et al. \cite{MR2105735}, Kurasov \cite{MR1397901}, and Kurasov and Boman\cite{MR1443392}.
In 1999 Savchuk and Shkalikov \cite{MR1756602} stirred great interest in Schr\"odinger equations with potentials in $W^{-1,2}_{\rm\loc}$.
However, it was pointed out by Eckhardt et al. \cite{MR3046408} in 2013 that such equations can be cast as first-order $2\times2$-systems whose coefficients are locally integrable.\footnote{This paper provides also a thorough study of the subject's history.}
Eckhardt and Teschl \cite{MR3095152} considered a Sturm-Liouville equation where the coefficients are distributions of order $0$.
As a system their equation reads $Ju'+qu=wf$ where $J=\sm{0&-1\\1&0}$, $q=\sm{\chi&0\\0&-\varsigma}$, and $w=\sm{\rho&0\\0&0}$.
This approach covers both the Krein string ($\chi=0$ and $\varsigma=1$) as well as the $\delta'$-interaction ($\chi=0$, $\varsigma=1+\beta\delta_0$, and $\rho=1$ in the simplest case).

Another strand of history begins with Chebyshev \cite{Chebyshev1873} in 1873 who raised a question later known as a moment problem.
This line of investigation led to inquiries on difference equations and their spectral theory and in particular to the study of orthogonal polynomials.
In his 1964 book \cite{MR0176141} Atkinson emphasizes the many similarities between differential and difference equations.
He recognizes that, in his words, ``neither the difference equation nor the differential equation provides a fully adequate framework for the topic of boundary value problems''.
Indeed, in the latter part of the book Atkinson lays the foundation for a unified approach by writing the differential/difference equation as a system of integral equation where integrals are to be viewed as matrix-valued Riemann-Stieltjes integrals.

It is the goal of the present paper to realize the plan envisioned by Atkinson employing Lebesgue-Stieltjes measures.
As explained already by Atkinson the admission of measures with discrete components as coefficients in the differential equation causes new problems, namely that it is not clear how to continue a solution when one reaches a point carrying mass unless one restricts oneself to a particular kind of functions of locally bounded variation, say left-continuous or right-continuous ones.
Consider, as a paradigm, the equation $u'=R\delta_0 u$ where $R$ is a fixed matrix in $\bb C^{n\times n}$.
Atkinson requires $R^2=0$ so that $(\id\pm R)^{-1}=\id\mp R$.
Then it does not matter whether one looks for left- or right-continuous solutions of initial value problems (except in the point $0$ itself).
In \cite{MR3095152} Eckhardt and Teschl require that the support of the discrete part of $\varsigma$ does not intersect the corresponding sets for $\chi$ or $\rho$.
This is precisely Atkinson's condition.

We are able to relax Atkinson's condition by considering balanced\footnote{We call a function of locally bounded variation balanced if it is the average of the corresponding left- and right-continuous variants.} solutions of our differential equations.
The reason for this is the integration by parts formula for functions of locally bounded variation which takes its usual form only when the jumps of the two functions under consideration do not coincide (essentially Atkinson's condition) or else when they are both balanced. This is explained in detail in Section \ref{sde.1}.
Nevertheless, the existence and uniqueness theorem for solutions of initial value problems of the equation $Ju'+qu=w(\la u+f)$ may fail for $\la\in\Lambda$, a set we require to be at most countable.

Even in the case of constant coefficients the treatment of the equation $Ju'+qu=wf$ may offer two difficulties which do not occur for a Sturm-Liouville equation.

The first is that the equation may not give rise to a linear operator but only to a linear relation.
A spectral theory for linear relations was first developed by Arens \cite{MR0123188} and first applied to the system $Ju'+qu=wf$ (with locally integrable coefficients) by Orcutt \cite{Orcutt}.
Langer and Textorius \cite{MR516415,MR751188,MR824211} developed an approach to the spectral theory of relations using Krein's directing functionals.
Bennewitz, in unpublished lecture notes \cite{Bennewitz-sths}, also addressed the subject.
We are, in various places, indebted to his exposition.

The second problem is that there may be non-trivial solutions of $Ju'+qu=0$ which are equivalent to $0$ in the appropriate Hilbert space.
Many people assumed in their work the definiteness condition which posits that this does not happen.
The first to consider the consequences when the definiteness condition is violated was Kac \cite{MR725424,zbMATH02007668} while working on $2\times 2$-systems.
Other contributions in this direction are due to Kogan and Rofe-Beketov \cite{MR0454141} and Lesch and Malamud \cite{MR1964480} who investigate the deficiency indices of the minimal operators relations associated with the equation $Ju'+qu=wf$.

Various aspects of systems of first-order differential equations with locally integrable coefficients and their spectral theory have been investigated by many people (with or without employing relations or the stipulation of the definiteness condition).
It is impossible to give an adequate overview in a few lines but we mention here papers by Binding and Volkmer \cite{MR1824087}, Dijksma, Langer, and de Snoo \cite{MR1017665,MR1251013}, Hinton and Shaw \cite{MR621248,MR688797,MR644777}, Mogileveskii \cite{MR2988011,MR3310511,MR3355783,MR3430753}, and Volkmer \cite{MR2135259} as well as the books by Arov and Dym \cite{MR2985757}, Gohberg and Krein \cite{MR0264447}, Sakhnovich, Sakhnovich, and Roitberg \cite{MR3098432}, and Weidmann \cite{MR923320}.

\begin{comment}
$(a,b)=(0,1)$, $n=2$ and $w=\sm{1&0\\ 0&0}$
\begin{enumerate}
  \item $J=\sm{0&1\\ -1&0}$ and $q=\sm{0&0\\ 0&-1}$. In this case $T_\mx$ is an operator and the definiteness condition holds.
  \item $J=\sm{0&1\\ -1&0}$ and $q=\sm{0&0\\ 0&0}$. In this case $T_\mx$ is not an operator and the definiteness condition does not hold.
  \item $J=i \id_2$ and $q=\sm{0&0\\ 0&0}$. In this case $T_\mx$ is an operator but the definiteness condition does not hold.
  \item $J=i \id_2$ and $q=\sm{0&1\\ 1&0}$. In this case $T_\mx$ is not an operator but the definiteness condition holds.
\end{enumerate}
\end{comment}

We close this introduction with several comments about our notation.
We denote identity operators by $\id$.
The transpose and conjugate transpose of $A\in\bb C^{m\times n}$, a matrix with $m$ rows and $n$ columns, are denoted by $A^\top$ and $A^*$, respectively.
We think of vectors in $\bb C^n=\bb C^{n\times1}$ as columns so that $x^*y$ is the scalar product in $\bb C^n$.
This scalar product is linear in the second entry, a convention which will be in force for all scalar products (generally denoted by $\<\cdot,\cdot\>$) occurring in this paper.
The function $A\mapsto |A|_1=\sum_{j=1}^m\sum_{k=1}^n |A_{j,k}|$ is a norm on the space of $m\times n$-matrices which we will sometimes use.
If $\mc H$ is a Hilbert space and $A=(A_1,...,A_\ell)\in\mc H^{1\times\ell}$ and $B=(B_1, ..., B_m)\in\mc H^{1\times m}$ we define $\<A,B\>\in \bb C^{\ell\times m}$ by $\<A,B\>_{j,k}=\<A_j,B_k\>$.
In particular, if $m=1$, then $\<A,B\>$ is a column in $\bb C^\ell$.
The direct sum of two subspaces $S$ and $T$ of a given Hilbert space $\mc H$ with trivial intersection is denoted by $S\dot+T$.
When we write $S\oplus T$ instead of $S\dot+T$, we assume that $S$ and $T$ are closed and orthogonal to each other.
The orthogonal complement of a subset $T$ of $\mc H$ is denoted by $T^\perp$ or $\mc H\ominus T$.
The characteristic function of a set $Y$ is denoted by $\chi_Y$.
We also use, on occasion, the $\sgn$ function which is $-1$ on the negative real axis, $0$ at zero, and $+1$ on the positive real axis.

\section{Differential equations with distributional coefficients}\label{sde}
\subsection{Existence and uniqueness} \label{sde.1}
In this section we investigate existence and uniqueness of solutions of first-order systems whose coefficients are distributions of order $0$.
Suppose $r\in \mc D^{\prime0}(\iOmega)^{n\times n}$ and $g\in\mc D^{\prime0}(\iOmega)^{n}$.
If the components of $u:\iOmega\to\bb C^n$ are of locally bounded variation, then the components of $u'$ and, using the definition made in equation \eqref{170414.1}, the components of $ru$ are distributions of order $0$.
Thus we may state the equation
$$u'=ru+g$$
seeking solutions in $\BVl(\iOmega)^n$.

There is some freedom in defining functions of bounded variation at points where they are discontinuous.
If $u\in\BVl(\iOmega)$, it is common to single out its left-continuous version $u^-$ and its right-continuous version $u^+$.
However, we will be particularly interested in {\em balanced}\index{balanced} functions, functions whose values are the average of left- and right-hand limits.
These will be denoted by $u^\#=(u^++u^-)/2$.
Our motivation for this is the resulting integration by parts formula and, of course, the fact that integration by parts will be a central tool later on.
Assume that $u,v\in\BVl(\iOmega)$ satisfy $u=t u^++(1-t)u^-$ and $v=t v^++(1-t)v^-$ for some fixed parameter $t$
(here $t=0$ yields left-continuous functions while $t=1$ yields right-continuous ones).
According to Lemma \ref{ibyp}
\begin{equation}\label{ibypt}
\int_{[c,d]} (udv+vdu)=(uv)^+(d)-(uv)^-(c)+(2t-1)\int_{[c,d]}(v^+-v^-)du
\end{equation}
whenever $[c,d]\subset\iOmega$.
We see from this that the choice of $t$ is irrelevant unless points of discontinuity of $u$ and $v$ coincide.
Since we do not want to rule out this possibility, choosing $t=1/2$ ensures that $\int_{[c,d]} (udv+vdu)$ depends only on the behavior of $u$ and $v$ near $c$ and $d$.

For any distribution $r\in \mc D^{\prime0}(\iOmega)^{m\times n}$ we define the function $\Delta_r:\iOmega\to \bb C^{m\times n}$ by
$$\Delta_r(x)=R^+(x)-R^-(x)$$
when $R$ is an antiderivative of $r$.
Of course, $\Delta_r(x)=0$ except on a countable set.

The existence and uniqueness theorem (Theorem \ref{EUIVP} below), which is at the base of our work, is due to \cite{BBW}.
It relies on the following result by Bennewitz \cite{MR1004432}.
\begin{thm}\label{BEU}
Let $x_0$ be a point in $\iOmega$.
Then the initial value problem $u'=ru+g$, $u(x_0)=u_0\in\bb C^n$ has a unique left-continuous solution $u\in\BVl([x_0,b))^n$.
\end{thm}

Bennewitz's proof can easily be adapted to show existence and uniqueness of a right-continuous solution of the initial value problem $u'=ru+g$, $u(x_0)=u_0\in\bb C^n$ in $\BVl((a,x_0])^n$.

\begin{thm}\label{EUIVP}
Suppose $r\in \mc D^{\prime0}(\iOmega)^{n\times n}$, $g\in\mc D^{\prime0}(\iOmega)^{n}$ and that the matrices $\id\pm\Delta_r(x)/2$ are invertible for all $x\in\iOmega$.
Let $x_0$ be a point in $\iOmega$.
Then the initial value problem $u'=ru+g$, $u(x_0)=u_0\in\bb C^n$ has a unique balanced solution $u\in\BVl^\#(\iOmega)^n$.
\end{thm}

\begin{proof}
Define $\tilde r=r(\id-\Delta_r/2)^{-1}$, $\tilde g=(\id-\Delta_r/2)^{-1}g$, and $\tilde u_0=(\id-\Delta_r(x_0)/2)u_0-\Delta_g(x_0)/2$.
Note that $\tilde r$ and $\tilde g$ are again distributions of order $0$ according to Theorem \ref{T:A.4}.
Then, with the aid of Theorem \ref{BEU}, we obtain a unique left-continuous solution $u_r$ of the initial value problem $u'=\tilde r u+\tilde g$, $u(x_0)=\tilde u_0$ on the interval $[x_0,b)$.
Similarly, we obtain a right-continuous solution $u_\ell$ of another modified problem on the interval $(a,x_0]$.
This time we choose $\tilde r=r(\id+\Delta_r/2)^{-1}$, $\tilde g=(\id+\Delta_r/2)^{-1}g$, and $\tilde u_0=(\id+\Delta_r(x_0)/2)u_0+\Delta_g(x_0)/2$.
Defining $u$ as $u_r^\#$ on $[x_0,b)$ and as $u_\ell^\#$ on $(a,x_0]$ gives then the desired balanced solution of the initial value problem $u'=ru+g$, $u(x_0)=u_0$.

We complete the proof by providing some of the details showing that $u_r^\#$ is a balanced solution of $u'=ru+g$, $u(x_0)=u_0$ on $[x_0,b)$
while we skip the details for the corresponding claim on $u_\ell$.
To emphasize that $u_r$ is left-continuous we will write $u_r^-$ for $u_r$ in the sequel.
Since $u_r^{-\prime}-\tilde ru_r^--\tilde g$ is the zero measure on $[x_0,b)$, we obtain $0$ upon computing the measure of any singleton $\{x\}$.
This and the identity
\begin{equation}\label{170410.1}
\id+\frac12\Delta_r(x)(\id-\frac12\Delta_r(x)\big)^{-1}=\big(\id-\frac12\Delta_r(x)\big)^{-1}
\end{equation}
imply
$$u_r^\#(x)=\big(\id-\frac12\Delta_r(x)\big)^{-1}\big(u_r^-(x)+\frac12\Delta_g(x)\big)$$
which shows, for one thing, that $u_r^\#(x_0)=u_0$.
Since the functions $u_r^\#$ and $u_r^-$ are equal away from a countable set, we have $u_r^{\#\prime}=u_r^{-\prime}$.
Also note that
\begin{equation}\label{170410.2}
r\big(\id-\frac12\Delta_r\big)^{-1}\Delta_g=\Delta_r\big(\id-\frac12\Delta_r\big)^{-1}g.
\end{equation}
Thus
\begin{multline*}
u_r^{\#\prime}-ru_r^\# -g=u_r^{-\prime}-r(\id-\Delta_r/2)^{-1}(u_r^-+\Delta_g/2)-g=u_r^{-\prime}-\tilde ru_r^--\tilde g=0.
\end{multline*}
This completes our proof.
\end{proof}

\begin{remark}
A similar proof shows that unique left-continuous (or a right-contin\-uous) solutions for initial value problems exist on all of $(a,b)$ provided that $\id+\Delta_r(x)$ (or $\id-\Delta_r(x)$) is always invertible.
Since Atkinson's condition $\Delta_r^2=0$ implies $(\id+\Delta_r)(\id-\Delta_r)=\id$ we obtain in this case the existence and uniqueness of both a left- and right-continuous solution.
In fact, these solutions agree away from their points of discontinuity.
We emphasize that Schwabik et al. \cite{MR542283} have a similar result (Theorem III.3.1) in terms of Perron-Stieltjes integrals where they also require that the matrices $\id\pm\Delta_r(x)$ are invertible for all $x\in(a,b)$.
\end{remark}

\begin{remark} \label{R:2.4}
If, for some $c\in\iOmega$, the left-continuous antiderivatives $R$ and $G$ of $r$ and $g$ are of bounded variation on $(a,c)$ and $u'=ru+g$, then $u$ is of bounded variation on $(a,c)$.
Thus $u$ has a limit at $a$ and one may solve the initial value problem with the initial condition posed at $a$; even if $a=-\infty$.
Corresponding statements hold for $b$.
\end{remark}

The following simple example is perhaps instructive.
Let $\iOmega=\bb R$, $n=1$, $r=\alpha\delta_0$ and $g=0$.
Any solution of the differential equation must be constant to the left and right of $0$.
Denoting these constants by $u_\ell$ and $u_r$, respectively, we get $u_r-u_\ell=\alpha u(0)$.
Now suppose we have an initial value $u_0$ at $x_0<0$.
While, in this case, the left-continuous, the right-continuous, and the balanced solution on $(-\infty,0)$ are all equal to the constant $u_0$, we get on the interval $(0,\infty)$ that $u_r=(1+\alpha)u_0$, $u_r=u_0/(1-\alpha)$, and $u_r=(2+\alpha)u_0/(2-\alpha)$, respectively.
In accordance with Theorem \ref{BEU}, a left-continuous solution exists regardless of the value of $\alpha$.
However, if $\alpha=1$ and $u_0\neq0$, then there is no right-continuous solution.
If $\alpha=1$ and $u_0=0$, we may choose anything for $u_r$ to obtain a right-continuous solution.
Similar comments hold for $\alpha=2$ and balanced solutions.
Note that Atkinson's condition can never be satisfied when $n=1$ and $R$ has discontinuities.

Since linear combinations of balanced solutions are again balanced solutions and since the initial value $u_0$ may be chosen freely in an $n$-dimensional space we get the following corollary.
\begin{cor} \label{C4.1.4}
Let $r$ be as in the previous theorem, in particular, assume $\id\pm\Delta_r/2$ always invertible.
Then the set of balanced solutions to the homogeneous equation $u'=ru$ in $\iOmega$ is an $n$-dimensional vector space.
\end{cor}

Note, however, that the stipulation on the matrices $\id\pm\Delta_r/2$ is important.
Indeed, if $\iOmega=(0,3)$, $n=1$ and $r=-2\delta_1+2\delta_2$, then any function $u$ which is constant on both $(0,1)$ and $(2,3)$ and zero on $(1,2)$ provides a balanced solution of $u'=ru$.
Thus, in this case, the space of solutions is two-dimensional.

\subsection{Variation of constants}
If $u_h$ and $u_p$ are balanced solutions of $u'=ru$ and $u'=ru+g$, respectively, then $u_h+u_p$ is also a balanced solutions of $u'=ru+g$.
Moreover, if $v$ is any balanced solution of $u'=ru+g$, then there is a balanced solution $u_h$ of $u'=ru$ such that $v=u_h+u_p$.
Thus, Corollary \ref{C4.1.4} describes implicitly also the manifold of solutions of the inhomogeneous equation $u'=ru+g$.
The following lemma is a generalization of the variation of constants formula, which provides a particular solution of the inhomogeneous equation, namely the one with zero initial conditions.
A {\em fundamental matrix}\index{fundamental matrix} $U$ for $u'=ru$ is an element of $\BVl^\#(\iOmega)^{n\times n}$ such that each column is a balanced solution of $u'=ru$ and $\det U(x)\neq0$ for some fixed $x\in\iOmega$.
Note that
\begin{equation}\label{170616.1}
U^\pm=(\id\pm\frac12\Delta_r)U
\end{equation}
and, defining the function $H$,
\begin{equation}\label{170701.1}
H=(\id-\frac12\Delta_r)U^+=(\id+\frac12\Delta_r)U^-.
\end{equation}
\begin{lem}\label{varconst}
Let $r$, $g$ and $x_0$ be as in Theorem \ref{EUIVP}, in particular, assume $\id\pm\Delta_r/2$ always invertible.
If $U$ is a fundamental matrix for $u'=ru$, then $\det U(x)$, $\det U^-(x)$, and $\det U^+(x)$ are different from zero for all $x$ in $\iOmega$.

Suppose $u^\pm$ are given by
\begin{equation}\label{150811.1}
u^-(x)=U^-(x)\big(-u_0^-+\int_{[x_0,x)} H^{-1}g\big), \quad x\geq x_0
\end{equation}
and
\begin{equation}\label{150811.2}
u^+(x)=U^+(x)\big(u_0^+-\int_{(x,x_0]} H^{-1}g\big), \quad x\leq x_0
\end{equation}
where $2u_0^{\pm}=U^\pm(x_0)^{-1} \Delta_g(x_0)$.
Then $u^\#$ is a balanced solution of the initial value problem $u'=ru+g$, $u(x_0)=0$.
Conversely, given a balanced solution $u$ of this initial value problem, formulas \eqref{150811.1} and \eqref{150811.2} hold.
\end{lem}

\begin{proof}
If there is a non-trivial $C\in \bb C^n$ such that $U(x_1)C=0$, then $UC$ is a balanced solution of $u'=ru$ which vanishes in $x_1$ and must therefore vanish identically.
But this is impossible proving that $\det U$ is never $0$.
Equation \eqref{170616.1} shows that both, $U^+$ and $U^-$, also have non-vanishing determinants.

Differentiating formula \eqref{150811.1} using the product rule \eqref{170410.3} and the identities \eqref{170410.1} and \eqref{170410.2} shows that $u^-$ satisfies the initial value problem $u'=\tilde ru+\tilde g$, $u(x_0)=\tilde u_0$ where $\tilde r$, $\tilde g$, and $\tilde u_0$ are those given in the proof of Theorem \ref{EUIVP}.
Hence $u^\#$ solves $u'=ru+g$, $u(x_0)=0$.
\end{proof}

\subsection{Dependence of the coefficients on a parameter}\label{sde.3}
We now consider the case where $r$, $g$, and $u_0$ depend analytically\footnote{See Appendix \ref{ADM.4} for the definition of \textit{analytic} in the context of distributions.} on a parameter.
Derivatives with respect to $\la$ are denoted by a dot set above the symbol representing the function in question.

\begin{thm}\label{T:2.3.1}
Let $\Omega$ be an open set in $\bb C$.
Suppose that $u_0:\Omega\to\bb C^n$, $r:\Omega\to\mc D^{\prime0}((a,b))^{n\times n}$, and $g:\Omega\to\mc D^{\prime0}((a,b))^n$ are analytic in $\Omega$.
Furthermore, assume that $\phi_\pm(x,\la)=\id\pm \Delta_{r(\la)}(x)/2$ are invertible whenever $x\in(a,b)$ and $\la\in\Omega$.
Let $u(.,\lambda)\in \BVl((a,b))^{n}$ be the unique balanced solution for the initial value problem $u'=r(\lambda)u+g(\lambda)$, $u(x_0,\lambda)=u_0(\lambda)$.
Then $u(x,.)$ is analytic in $\Omega$ for each $x\in(a,b)$.
\end{thm}

\begin{proof}
We begin by pointing out that the analyticity of one of $u^-(x,\cdot)$, $u^+(x,\cdot)$, and $u(x,\cdot)$ implies the analyticity of the others because of the equations
$u^+(x,\la)-u^-(x,\la)=\Delta_{r(\la)}(x)u(x,\la)+\Delta_{g(\la)}(x)$ and $u^+(x,\la)+u^-(x,\la)=2u(x,\la)$.
Therefore the function $z=u^-(x_0,\cdot)$ is analytic.
For a given $\la_0\in\Omega$ we will show that $u^-(x,.)$ is differentiable at $\la_0$ for each $x\in[x_0,b)$.
A similar argument can be used when $x\in(a,x_0]$.

Let $v\in \BVl((a,b))^{n}$ be the balanced solution of the initial value problem
$$y'=r(\lambda_0)y+\dot r(\lambda_0)u(.,\lambda_0)+\dot g(\lambda_0)$$
$$y(x_0)=\phi_-(x_0,\la_0)^{-1}\big(2\dot{z}(\la_0)+\Delta_{\dot{r}(\la_0)}(x_0)u_0(\la_0)+\Delta_{\dot{g}(\la_0)}(x_0)\big)/2.$$
Here the initial condition was chosen so that $v^-(x_0)=\dot{z}(\la_0)$.
For $\la\neq\la_0$ in $\Omega$ let $\check{u}(.,\lambda)=(u(.,\la)-u(.,\la_0))/(\la-\la_0)-v$.
Then $\check{u}(.,\la)\in \BVl((a,b))^{n}$ is a balanced solution of
$$y'=r(\la)y+(r(\la)-r(\la_0))v+\check{r}(\la)u(.,\la_0)+\check{g}(\la),$$
satisfying $\check{u}^-(x_0,\la)=\check{z}(\la)$
when
\begin{multline*}
\check{r}(\la)=\frac{r(\la)-r(\la_0)}{\la-\la_0}-\dot r(\la_0),\; \check{g}(\la)=\frac{g(\la)-g(\la_0)}{\la-\la_0}-\dot g(\la_0), \text{ and }\\
\check z(\la)=\frac{z(\la)-z(\la_0)}{\la-\la_0}-\dot z(\la_0).
\end{multline*}
This implies that for all $s\in[x_0,b)$
$$\check{u}^-(s,\lambda)=\check z(\lambda)+\int_{[x_0,s)}\Big(r(\lambda)\check{u}(.,\lambda)+(r(\lambda)-r(\lambda_0))v+\check{r}(\la)u(.,\la_0)+\check{g}(\la)\Big).$$
We prove now that $\lim_{\lambda\rightarrow\lambda_0}\vert\check{u}^-(x,\la)\vert_1=0$ for any given $x\in(x_0,b)$.
Pick $c\in(x,b)$ and choose an open set $O$ containing $\la_0$ whose closure is contained in $\Omega$.
Choose $M>0$ such that $\vert v(s)\vert_1\leq M$, $\vert u(s,\la_0)\vert_1\leq M$, and, utilizing Theorem \ref{T:A.4}, $\vert\phi_-(s,\la)^{-1}\vert_1\leq M$
for all $s\in[x_0,c]$ and all $\la\in O$.
Then, for all $s\in[x_0,c)$ and all $\la\in O\setminus\{\la_0\}$, we have
\begin{equation}\label{0101011}
\vert\check{u}^-(s,\la)\vert_1\leq\alpha(\la)+ \int_{[x_0,s)}\vert\check{u}(.,\la)\vert_1dV_{R(.,\la)},
\end{equation}
where
$$\alpha(\la)=\vert\check z(\la)\vert_1+M \Var_{R(.,\la)-R(.,\la_0)}([x_0,c])+M\Var_{\check{R}(.,\la)}([x_0,c])+\Var_{\check{G}(.,\la)}([x_0,c])$$
and $R$, $\check R$, and $\check G$ are left-continuous antiderivatives of $r$, $\check r$, and $\check g$, respectively.
From
$$\check{u}(t,\la)=\frac12\phi_-(t,\la)^{-1}\big(2\check{u}^-(t,\la)+\Delta_{r(\la)-r(\la_0)}(t)v(t)+\Delta_{\check{r}(\la)}(t)u(t,\la_0)+\Delta_{\check{g}(\la)}(t)\big)$$
we obtain with help of \eqref{180320.1}
\begin{equation}\label{010101012}
\vert\check{u}(t,\la)\vert_1\leq\beta(\la)+M\vert\check{u}^-(t,\la)\vert_1,
\end{equation}
for all $t\in[x_0,c)$ and all $\la\in O\setminus\{\la_0\}$
when
$$\beta(\la)=\frac M2\big(M\Var_{R(.,\la)-R(.,\la_0)}([x_0,c])+M\Var_{\check{R}(.,\la)}([x_0,c])+\Var_{\check{G}(.,\la)}([x_0,c])\big).$$
By combining \eqref{0101011} and \eqref{010101012} we obtain
$$\vert\check{u}^-(s,\la)\vert_1\leq\alpha(\la)+\beta(\la)\Var_{R(.,\la)}([x_0,c])+\int_{[x_0,s)} M\vert\check{u}^-(.,\la)\vert_1dV_{R(.,\la)},$$
for all $s\in[x_0,c)$ and all $\la\in O\setminus\{\la_0\}$.
Now Gronwall's inequality (see, e.g., Lemma 1.3 of \cite{MR1004432}) gives
$$\vert\check{u}^-(x,\la)\vert_1\leq\big(\alpha(\la)+\beta(\la)\Var_{R(.,\la)}([x_0,c])\big)\exp(M\Var_{R(.,\la)}([x_0,c])).$$
Since $\lim_{\la\rightarrow\la_0}\alpha(\la)=\lim_{\la\rightarrow\la_0}\beta(\la)=0$ and since $\Var_{R(.,\la)}([x_0,c])$ is bounded in $O$ it follows that the derivative of $u^-(x,\cdot)$ at $\la_0$, being equal to $v^-(x)$, exists so that our claim is proved.
\end{proof}

We now specialize to the case needed later on.
Suppose $J$ is a constant and invertible matrix, $q$ and $w$ are in $\mc D^{\prime0}(\iOmega)^{n\times n}$, and $f$ a function for which $wf$ is in $\mc D^{\prime0}(\iOmega)^n$.
Then let $r=J^{-1}(\la w-q)$, $g=J^{-1}wf$ and $u_0$ constant.
We introduce the set $\Lambda=\bigcup_{x\in(a,b)}\Lambda_x$ where
$$\Lambda_x=\{\la\in\bb C: \det(\phi_+(x,\la))\det(\phi_-(x,\la))=0\}.$$
Thus, according to Theorem \ref{EUIVP}, the initial value problem $Ju'+qu=w(\la u+f)$, $u(x_0)=u_0$ has a unique balanced solution unless $\la\in\Lambda$.

If $0\not\in\Lambda$, Theorem \ref{T:A.4} gives that $\bigcup_{x\in[s,t]}\Lambda_x$ is a closed set of isolated points, even though $\Lambda$ may not be.
We can now apply Theorem \ref{T:2.3.1} with $\Omega=\bb C\setminus\bigcup_{x\in[s,t]}\Lambda_x$ and $(a,b)=(s,t)$ to obtain that the solutions of the initial value problem $Ju'+qu=w(\la u+f)$, $u(x_0)=u_0$ are analytic as functions of $\la\in\Omega$.

\section{Minimal and maximal relations}\label{sop}
As mentioned in the introduction our goal in this paper is to study the spectral theory of the first-order system
$$Ju'+qu=wf$$
on the interval $\iOmega$.
For the remainder of the paper we require the following hypothesis to be satisfied.
\begin{hyp}\label{H:4.1}
$J$ is a constant, invertible and skew-Hermitian matrix.
Both $q$ and $w$ are in $\mc D^{\prime0}(\iOmega)^{n\times n}$, $w$ is non-negative and $q$ Hermitian.
Moreover, the matrices $2J\pm\Delta_q(x)$ are invertible for all $x\in(a,b)$.
\end{hyp}

If $f\in \mc L^2(w)$, each term in the equation $Ju'+qu=wf$ is a distribution of order $0$ and thus balanced solutions for the associated initial value problem always exist.

We note here that, for $n=1$, the number $2J\pm\Delta_q$ can never be zero while Atkinson's condition can never hold.
If $n>1$, the validity of Atkinson's condition implies that $2J\pm\Delta_q$ is invertible.
In other words, our constraint is considerably less restrictive than Atkinson's.

We now define the maximal and minimal relations in $L^2(w)\exsum L^2(w)$ associated with the system $Ju'+qu=wf$.
It turns out that one has to be more careful about distinguishing between the classes of functions in $L^2(w)$ and their representatives\footnote{As explained in the Appendix \ref{ADM.3} we denote the space of representatives of elements in $L^2(w)$ by $\mc L^2(w)$.}.
Therefore we define first the spaces
$$\mc T_\mx=\{(u,f)\in \mc L^2(w)\exsum \mc L^2(w): u\in\BVl^\#(\iOmega)^n, Ju'+qu=wf\}$$
and
$$\mc T_\mn=\{(u,f)\in \mc T_\mx: \text{$\supp u$ is compact in $\iOmega$}\}.$$
The maximal relation is then
$$T_\mx=\{([u],[f])\in L^2(w)\exsum L^2(w):(u,f)\in\mc T_\mx\}$$
and the minimal relation is
$$T_\mn=\{([u],[f])\in L^2(w)\exsum L^2(w):(u,f)\in\mc T_\mn\}.$$
Here (and elsewhere) we choose, as is customary, brevity over precision: whenever we have a pair $([u],[f])$ in $T_\mx$ or $T_\mn$ we assume that $u$ is an element of $\BVl^\#(\iOmega)^n$.

Note that $T_\mx$ may not be an operator.
Suppose, for example, that $n=2$, $\iOmega=(0,1)$, $J=i\id$, $q=\sm{0&1\\ 1&0}$ and $w=\sm{1&0\\ 0&0}$.
Then $u=(0,1)^\top$ and $f=(1,0)^\top$ satisfy $Ju'+qu=wf$ so that $([u],[f])\in T_\mx$.
Since $\|u\|=0$, but $\|f\|=1$, it follows that $T_\mx$ is not an operator.
Lest we create the impression that the absence of the definiteness condition causes the necessity of dealing with relations rather than operators, we emphasize that the definiteness condition is satisfied in this example.

\begin{comment}
Suppose, for example, that $n=1$, $\iOmega=\bb R$, $J=i$, $q=\beta\delta_0$, and $w=\delta_0$.
Then each element of $L^2(w)$ is a class of functions whose values agree at the point $0$.
As a consequence $L^2(w)$ is one-dimensional.
Solutions of the differential equation $Ju'+qu=wf$ have to be constant on $(-\infty,0)$ and $(0,\infty)$, respectively.
Denoting these constants by $u_\ell$ and $u_r$ we obtain
$$\begin{pmatrix}1 & -1 \\ 1 & 1 \end{pmatrix} \begin{pmatrix}u_r \\ u_\ell \end{pmatrix} = \begin{pmatrix}i\beta & -i \\ 2 & 0\end{pmatrix} \begin{pmatrix}u(0) \\ f(0) \end{pmatrix}.$$
Since the left matrix is invertible we obtain a solution for any choice of $(u(0),f(0))$, \ie, $T_\mx=L^2(w)\exsum L^2(w)$.
For elements of $T_\mn$ we must have $u_\ell=u_r=0$ and thus, since the right matrix is also invertible $u(0)=f(0)=0$, \ie, $T_\mn=\{(0,0)\}$.
Note that $T_\mn^*=T_\mx$ and $T_\mx^*=T_\mn$.
\end{comment}

In Section \ref{sde.3} we introduced the set
$$\Lambda=\{\la\in\bb C: \exists x\in\iOmega: \det(2J\pm\Delta_{\la w-q}(x))=0\}$$
of points where the initial value problem $Ju'+qu=w(\la u+f)$, $u(x_0)=u_0$ may fail to have a balanced solution or may have infinitely many such solutions.
Under the present hypotheses we have that $\la\in\Lambda$ if and only if $\ol\in\Lambda$.
Recall, however, that $0$ is never in $\Lambda$ so that the initial value problem $Ju'+qu=wf$, $u(x_0)=u_0$ has always a unique balanced solution.
Subsequently, we will always tacitly assume that solutions of $Ju'+qu=\la w u+wf$ are balanced unless explicitly stated otherwise.

The next step is to calculate the adjoint of $T_\mn$.
In order to do this, we shall first show that $u^{\pm*}Jv^\pm$ plays the role of the Wronskian of $u$ and $v$ and then adapt the variation of constants formula given in Lemma \ref{varconst} to the present situation.

\begin{lem}\label{T4.1.4a}
If $\la\not\in\Lambda$, $Ju'+qu=\ol wu$, and $Jv'+qv=\la wv$, then
$$u^{+*}Jv^+=u^{-*}Jv^-$$
is constant on $\iOmega$.
\end{lem}

\begin{proof}
A variant of the integration by parts formula \eqref{ibypt} for $t=1/2$ gives
$$(u^{-*}Jv^-)(x) -(u^{-*}Jv^-)(y)=\int_{[y,x)}\big(-(Ju')^*v+u^*Jv'\big).$$
But, using the differential equation and the fact that $q$ and $w$ are Hermitian, the right-hand side is $0$.
Thus $u^{-*}Jv^-$ and, by a similar argument, $u^{+*}Jv^+$ are constant.
These constants have to be identical since $u^-=u^+$ and $v^-=v^+$ outside a countable set.
\end{proof}

\begin{lem}\label{varconst2}
For each $\la\not\in\Lambda$ let $U(\cdot,\la)$ be the fundamental matrix for $Ju'+qu=\la wu$ that satisfies $U(x_0,\la)=\id$ for some fixed $x_0\in(a,b)$ which is a point of continuity for both $Q$ and $W$, the antiderivatives of $q$ and $w$.

Suppose $f\in \mc L^1_\loc(w)$ and the functions $u^\pm$ are given by
\begin{equation}\label{150812.1}
u^-(x,\la)=U^-(x,\la)J^{-1}\int_{[x_0,x)} U(\cdot,\ol)^* wf, \quad x\geq x_0
\end{equation}
and
\begin{equation}\label{150812.2}
u^+(x,\la)=-U^+(x,\la)J^{-1}\int_{(x,x_0]} U(\cdot,\ol)^*wf, \quad x\leq x_0.
\end{equation}
Then $u^\#$ is a balanced solution of the initial value problem $Ju'+qu=\la wu +wf$, $u(x_0)=0$.
Conversely, given a balanced solution $u$ of this initial value problem, the associated functions $u^\pm$ are given by formulas \eqref{150812.1} and \eqref{150812.2}.
\end{lem}

\begin{proof}
Let $r=J^{-1}(\la w-q)$ and $g=J^{-1}wf$.
Then $\Delta_r=J^{-1}(\la \Delta_w-\Delta_q)$.
From Lemma~\ref{T4.1.4a} we get $U^+(\cdot,\ol)^*JU^+(\cdot,\la)=J$.
This and the adjoint of equation \eqref{170616.1} for $\ol$ gives
\begin{multline*}
J=U(\cdot,\ol)^*(\id-\frac12(\la \Delta_w-\Delta_q) J^{-1})JU^+(\cdot,\la)\\
 =U(\cdot,\ol)^*J(\id-\frac12 J^{-1}(\la \Delta_w-\Delta_q))U^+(\cdot,\la).
\end{multline*}
In view of \eqref{170701.1} we have now $J=U(\cdot,\ol)^*JH(\cdot,\la)$, \ie, $H(\cdot,\la)^{-1}=J^{-1}U(\cdot,\ol)^*J$.
Use this in Lemma \ref{varconst}.
\end{proof}

We are now ready to prove that the maximal relation is the adjoint of the minimal relation, which is the cornerstone of all subsequent considerations.

\begin{thm}\label{T4.1.6}
$T_\mn^*=T_\mx$.
\end{thm}

\begin{proof}
To prove $T_\mx\subset T_\mn^*$ assume that $([v],[g])\in T_\mx$ and $([u],[f])\in T_\mn$.
We need to show that $\<u,g\>=\<f,v\>$.
Suppose $\supp u\subset(s,t)\subset[s,t]\subset\iOmega$.
Then, using integration by parts,
$$\<u,g\>=\int u^*(Jv'+qv)=\int ((Ju')^*v+u^*qv)=\int (wf)^*v=\<f,v\>.$$

To prove $T_\mn^*\subset T_\mx$ requires more effort.
Given $[s,t]\subset (a,b)$ we define the distribution $\tilde w=\chi_{(s,t)}w$.
Accordingly we have a relation $\tilde T_\mx$ associated with $J$, $q$, and $\tilde w$.
Let $\tilde T_0=\{([u],[f])\in\tilde T_\mx:\supp u\subset [s,t]\}$ and $K_0=\ker\tilde T_\mx$ and note that $\dim K_0\leq n$.
Then, as we will show below, $\ran \tilde T_0=L^2(\tilde w)\ominus K_0$.
Assume now that $([v],[g])\in T_\mn^*$.
Since $0\not\in\Lambda$ Theorem \ref{EUIVP} gives the existence of a balanced function $v_1$ of locally bounded variation satisfying $Jv_1'+qv_1=wg$.
If $([u],[f])\in \tilde T_0$, let $\tilde f=\chi_{(s,t)}f$ so that $\tilde w f=w\tilde f$.
Then $([u],[\tilde f])\in T_\mn$ so that $\<\tilde f,v\>=\<u,g\>$.
An integration by parts shows that
$$\<u,g\>=\int u^*wg=\int u^*(Jv_1'+qv_1)=\int (Ju'+qu)^*v_1=\int f^*\tilde w v_1.$$
Thus $\int f^*\tilde w(v-v_1)=0$, i.e., $[v-v_1]\in K_0$.
This implies that $[v]$ has a representative $v$ such that $Jv'+qv=wg$ on $(s,t)$.
Since this works for all intervals $[s,t]\subset (a,b)$ we see that $([v],[g])\in T_\mx$.

It remains to show that $\ran \tilde T_0=L^2(\tilde w)\ominus K_0$.
Let $[f]\in \ran \tilde T_0$ and $[k]\in K_0$.
Then $Ju'+qu=\tilde wf$ for some $u$ whose support is in $[s,t]$ and $k$ (chosen appropriately in $[k]$) satisfies $Jk'+qk=0$.
Integration by parts shows then
$$\int f^{\ast}\tilde wk=\int u^{\ast}(Jk'+qk)=0.$$
Conversely, let $f\in L^2(\tilde w)\ominus K_0$ and $U$ the fundamental matrix for $Ju'+qu=0$, which equals the identity matrix at some point of continuity for $Q$ smaller than $s$.
By Lemma \ref{varconst2} the balanced solution for the initial value problem $Jv'+qv=\tilde wf$, $v(s)=0$ vanishes to the left of $s$.
Since the columns of $U$ are representatives of elements in $K_0$, the same lemma shows, when $x>t$, that
$$v^-(x)=U^-(x)J^{-1}\int_{[s,t]} U^{\ast}\tilde wf=0$$
so that $\supp v\subset [s,t]$ and hence $f\in\ran\tilde T_0$.
\end{proof}

If $([v],[g])$ and $([u],[f])$ are in $T_\mx$ we have that $v^*wf-g^*wu=v^*Ju'+v^{\prime*}Ju$ is a finite measure on $\iOmega$ and
$$(v^*Ju)^-(d)=(v^*Ju)^-(c)+\int_{[c, d)}(v^*Ju'+v^{\prime*}Ju)$$
if $[c,d]\subset(a,b)$.
Thus, if $[c, d_n)$ is a non-decreasing sequence of intervals converging to $[c,b)$, we obtain the convergence of $(v^*Ju)^-(d_n)$ as $d_n\uparrow b$.
By a similar argument, $(v^*Ju)^+(c_n)$ converges as $c_n\downarrow a$.
We shall denote these limits by $(v^*Ju)^-(b)$ and $(v^*Ju)^+(a)$, respectively.
Note that these limits may depend on the representatives chosen.
However their difference does not, since
\begin{equation}\label{Lagrange}
(v^*Ju)^-(b)-(v^*Ju)^+(a)=\<v,f\>-\<g,u\>
\end{equation}
where the right-hand side depends only on the respective classes but not the representatives.
Equation \eqref{Lagrange} is called Green's formula or Lagrange's identity.

\section{Self-adjoint restrictions of \texorpdfstring{$T_\mx$}{Tmax} and their resolvents}\label{sopt}
According to Appendix \ref{ALR} the symmetric restrictions of $T_\mx$ are determined by the structure of the deficiency spaces of $T_\mn$,
\ie, the spaces $D_\la=\{([u],\la[u])\in T_\mx\}$.
In particular, according to Corollary \ref{C:B.3}, $\dim D_\la$ is independent of $\la$ as long as $\la$ remains in either the upper or the lower half-plane.
Thus the deficiency indices are defined as $n_\pm=\dim D_{\pm i}$ and, according to Theorem \ref{vonNeumann2}, self-adjoint restrictions of $T_\mx$ exist if and only if $n_+=n_-$.
Since the space of solutions of $Ju'+qu=\la wu$ is $n$-dimensional, at least when $\la\not\in\Lambda$, it follows that the deficiency indices are at most $n$.
However, they may be strictly smaller than $n$, since the norm of a non-trivial solution could be zero or infinite.

Theorem \ref{T:B.6} states that any self-adjoint restriction $T$ of $T_\mx$ is given as $T=\ker A$ where $A$ is a surjective linear operator from $T_\mx$ to $\bb C^{n_+}$ with the properties that $T_\mn\subset\ker A$ and $A\mc JA^*=0$ (recall that $\mc J:(u,f)\mapsto(f,-u)$).
Each component of this map $A$ is a bounded linear functional on $T_\mx$ which is determined by its values on $V=D_i\oplus D_{-i}$.
Thus, by Riesz's representation theorem, each component of $A$ is represented by an element of $V$, \ie, $A_j(u,f)=\<(v_j,g_j),(u,f)\>$ for $j=1, ..., n_+$ with $(v_j,g_j)\in V$.

Now note that, according to Lagrange's identity \eqref{Lagrange},
$$\<(v_j,g_j),(u,f)\>=\<g_j,f\>-\<-v_j,u\>=(g_j^*Ju)^-(b)-(g_j^*Ju)^+(a)$$
using the fact that $(g_j,-v_j)$ is also an element of $V$ and hence of $T_\mx$.
Therefore the self-adjoint restrictions of $T_\mx$ are determined by the behavior of $u\in\dom T$ near the boundary of the interval $(a,b)$ and the conditions
\begin{equation}\label{180211.1}
(g_j^*Ju)^-(b)-(g_j^*Ju)^+(a)=0
\end{equation}
are called boundary conditions.
It is useful to recall here that the left-hand side of equation \eqref{180211.1} can be evaluated by choosing representatives of $(v_j,g_j)$ and $(u,f)$.
We have the following theorem.
\begin{thm}\label{T:5.1}
Suppose the deficiency indices of $T_\mn$ satisfy $n_+=n_->0$.
The relation $T$ is a self-adjoint restriction of $T_\mx$ if and only if there are $n_+$ linearly independent elements $(v_1,g_1)$, ..., $(v_{n_+},g_{n_+})$ in $V$ such that $\<(v_k,g_k),(g_\ell,-v_\ell)\>=0$ for all $1\leq k,\ell\leq n_+$.
In this case $T$ is given as
$$T=\{(u,f)\in T_\mx: (g_j^*Ju)^-(b)-(g_j^*Ju)^+(a)=0\text{ for $j=1,...,n_+$}\}.$$
\end{thm}

\begin{proof}
First assume $(v_1,g_1)$, ..., $(v_{n_+},g_{n_+})$ are given with the stated properties.
Define $A:T_\mx\to\bb C^{n_+}$ by $A_j(u,f)=\<(v_j,g_j),(u,f)\>$ for $j=1, ..., n_+$.
Then $A^*:\bb C^{n_+}\to T_\mx$ is given by
$$A^*(y)=\sum_{j=1}^{n_+} y_j(v_j,g_j).$$
It follows, using Lagrange's identity \eqref{Lagrange}, that
\begin{equation}\label{171206.1}
(A\mc JA^*)_{k,\ell}=\<(v_k,g_k),(g_\ell,-v_\ell)\>=(g_k^*Jg_\ell)^-(b)-(g_k^*Jg_\ell)^+(a)
\end{equation}
vanishes for all $k$ and $\ell$.
Thus $A$ satisfies the hypotheses of Theorem \ref{T:B.6} so that $\ker A$ is a self-adjoint restriction of $T_\mx$.

Conversely, if $T$ is a self-adjoint restriction of $T_\mx$, then there is a map $A:T_\mx\to\bb C^{n_+}$ with the properties stated in Theorem \ref{T:B.6}.
Each component of this map is a bounded linear functional on $T_\mx$, \ie, according to Riesz's representation theorem, its $j$th component is given by $(u,f)\mapsto\<(v_j,g_j),(u,f)\>$.
Now employ identity \eqref{171206.1} the other way.
\end{proof}

Since $\<(v_j,g_j),(v_j,g_j)\>\neq0$ it can not happen that $(g_j^*Ju)^+(a)$ and $(g_j^*Ju)^-(b)$ both vanish for all $(u,f)\in \mc T_\mx$.
We have therefore three cases: (1) $(g_j^*Ju)^+(a)=0$ for all $(u,f)\in \mc T_\mx$, (2) $(g_j^*Ju)^-(b)=0$ for all $(u,f)\in \mc T_\mx$, and (3) for some $(u,f)\in \mc T_\mx$ both $(g_j^*Ju)^+(a)$ and $(g_j^*Ju)^-(b)$ are different from zero.
Accordingly there are three kinds of boundary conditions.
We have separated boundary conditions, if we have case (1) or (2) for every $j\in\{1,...,n_+\}$; we have coupled boundary conditions, if we have case (3) for every $j\in\{1,...,n_+\}$; and we have mixed boundary conditions otherwise.

We now have a closer look at the case where the norm of at least one non-trivial solution of $Ju'+qu=\la wu$ is zero, \ie, the case where the definiteness condition mentioned in the introduction is violated.
Since $\|u\|=0$ if and only if $wu$ is the zero distribution, we may state the definiteness condition as saying that the vector space
\begin{equation}\label{170817.1}
\mc L_0=\{u\in \BVl^\#(\iOmega)^n: Ju'+qu=0, wu=0\}
\end{equation}
is trivial, a requirement which we shall not pose in general.
More broadly, suppose $([u],[f])\in T_\mx$ and that there are $u,v\in[u]$ and $f,g\in[f]$ such that $Ju'+qu=wf$ and $Jv'+qv=wg$.
Then $J(u-v)'+q(u-v)=w(f-g)=0$ as well as $w(u-v)=0$, \ie, $u-v\in \mc L_0$.
It follows that, given an element $([u],[f])\in T_\mx$, we may find a unique $u\in\dom\mc T_\mx$ such that $Ju'+qu=wf$.

Let $\RN$ denote the matrix of Radon-Nikodym derivatives of $w$ with respect to $\tr w$.
Since $\RN\geq 0$ we get that$\int(u^*\RN u)\tr w=0$ and hence $u^*\RN u=0$ almost everywhere with respect to $\tr w$, if $u\in\mc L_0$.
If $\dim \mc L_0=n$ this implies $\RN=0$ and hence $w=0$.
In this case $L^2(w)=\{0\}$, a case which we shall henceforth ignore.

Now, fix a point $x_0\in(a,b)$ and define $N_0=\{u(x_0):u\in \mc L_0\}\subset \bb C^n$.
The existence and uniqueness theorem implies that $\dim N_0=\dim\mc L_0$.
Thus, if $([u],[f])\in T_\mx$, there is a unique balanced $v\in[u]$ such that $Jv'+qv=wf$ and $v(x_0)\in N_0^\perp$.
This leads to the following definition.
\begin{defn}\label{D:5.2}
Suppose $x_0$ is a point in $(a,b)$.
Then the evaluation operator $E:T_\mx\to\BVl^\#((a,b))^n$ assigns to each $([u],[f])\in T_\mx$ the unique balanced representative $v$ of $[u]$ satisfying $Jv'+qv=wf$ and $v(x_0)\in N_0^\perp$.
\end{defn}
We emphasize that $E$ depends on the choice of the point $x_0$.

Suppose that $Q$ and $W$ are of bounded variation on an interval $(c,d)$ containing $x_0$ and contained in (or equal to) $(a,b)$.
Then the balanced representatives of elements of $\dom T_\mx$ are also of bounded variation on $(c,d)$.
Thus the restriction of $E(([u],[f]))$ to $(c,d)$ is an element $u^\#$ of the Banach space $\BV^\#((c,d))^n$ which is equipped with the norm $\vertiii u=|u(x_0)|_1+\Var_{u}((c,d))$.
We then define the linear operator $E_{(c,d)}:T_\mx\to\BV^\#((c,d))^n$ by $E_{(c,d)}(([u],[f]))=u^\#$.

\begin{lem}\label{L:6.2}
Suppose $(c,d)$ is a subinterval of $(a,b)$ such that $Q$ and $W$ are of bounded variation on $(c,d)$.
Then $E_{(c,d)}:T_\mx\to\BV^\#((c,d))^n$ is a bounded operator.
\end{lem}

\begin{proof}
Since $E_{(c,d)}$ is defined on all of $T_\mx$ the closed graph theorem implies the claim if we can show that $E_{(c,d)}$ is a closed operator.
Thus assume that $([u_j],[f_j])$ converges to $([u],[f])$ in $T_\mx$ and that $E_{(c,d)}([u_j],[f_j])$ converges to $v$ in $\BV^\#((c,d))^n$.
We may assume that $E_{(c,d)}([u_j],[f_j]))$ and $E_{(c,d)}([u],[f])$ are the restrictions of $u_j$ and $u$, respectively, to the interval $(c,d)$.
We will, however, make no distinction in the notation.
Our goal is to show that $u=v$ on $(c,d)$.

First note that $u_j(x_0)\in N_0^{\perp}$ and $|u_j(x_0)-v(x_0)|_1\rightarrow0$ imply that $v(x_0)\in N_0^{\perp}$.
Next pick a point $t\in(c,d)$ such that $\Delta_q(t)=\Delta_w(t)=0$.
Let $U$ be the fundamental matrix of $Ju'+qu=0$ which satisfies $U(t)=\id$.
Lemma \ref{varconst2}, the variation of constants formula, gives
$$u_j^-(x)=U^-(x)\Big(u_j(t)+J^{-1}\int_{[t,x)}U^{\ast}wf_j\Big)$$
as long as $x\geq t$.
Taking the limit as $j\to\infty$ gives
$$v^-(x)=U^-(x)\Big(v(t)+J^{-1}\int_{[t,x)}U^{\ast}wf\Big)$$
since convergence in $\BV^\#((c,d))^n$ implies pointwise convergence and the integral may be considered as a vector of scalar products which are, of course, continuous.
Applying Lemma \ref{varconst2} in the reverse gives that $v$ is a balanced solution for $Jv'+qv=wf$ on $[t,d)$.
That it is also a solution on $(c,t]$ follows similarly.
Since $u$ satisfies the same equation we have that $u-v$ satisfies $Jy'+qy=0$.
Next we show $w(u-v)=0$ on $(c,d)$.

Now $\|u-u_j\|\rightarrow0$ implies that $\int_{(c,d)}(u-u_j)^{\ast}\RN(u-u_j) \tr w\rightarrow0$ where $\RN$ is the matrix of Radon-Nikodym derivatives of $w$ with respect to $\tr w$.
This entails that the integrand converges pointwise almost everywhere with respect to $\tr w$.
On the other hand we know that the pointwise limit is $(u-v)^{\ast}\RN(u-v)$.
It follows that $w(u-v)=0$ on $(c,d)$.
Since $(c,d)$ can be chosen arbitrarily large in $(a,b)$ this shows $u-v\in \mc L_0$.
Since $(u-v)(x_0)\in N_0^\perp\cap N_0=\{0\}$ we get $u=v$ on $(a,b)$.
\end{proof}

In the sequel an important role will be played by $P$, the orthogonal projection from $\bb C^n$ onto $N_0^\perp$.
Also, from now on, $U(\cdot,\la)$ is the fundamental matrix for $Ju'+qu=\la wu$ which satisfies $U(x_0,\la)=\id$ (when $\la\not\in\Lambda$) and we assume that $x_0$ is chosen in such a way that the antiderivatives of $q$ and $w$ are continuous at $x_0$.
The existence and uniqueness theorem implies then that the columns of $U(\cdot,\la)(\id-P)$ span $\mc L_0$.
In particular, $w U(\cdot,\la)(\id-P)=0$.

\begin{lem}\label{L:5.1}
For $\la\in\bb C\setminus\Lambda$ let $B$ be the set of all vectors $\int U(\cdot,\ol)^*wf$ for compactly supported $f\in L^2(w)$.
Then $B=N_0^\perp$.
\end{lem}

\begin{proof}
Since $w U(\cdot,\ol)=w U(\cdot,\ol)P$ and thus $U(\cdot,\ol)^*w=P U(\cdot,\ol)^*w$, the set $B$ is a subset of $N_0^\perp=\ran P$.
Now suppose $\alpha$ is an element of $N_0^\perp$ which is perpendicular to $B$, \ie, $\alpha^*\int U(\cdot,\ol)^* w f=0$ for all compactly supported $f\in L^2(w)$.
Then each component of $\alpha^*U(\cdot,\ol)^* w$ is the zero distribution as is each component of $wU(\cdot,\ol)\alpha$.
Thus $U(\cdot,\ol)\alpha\in\mc L_0$ so that $\alpha\in N_0$.
Hence $\alpha=0$.
\end{proof}

In the remainder of this section let $T$ be a particular self-adjoint restriction of $T_\mx$.
In Appendix \ref{ALR.2} we defined the resolvent set
$$\rho(T)=\{\la\in\bb C: \ker(T-\la)=\{0\}, \ran(T-\la)=L^2(w)\}$$
and the resolvent of $T$ at $\la$, \ie, the closed linear operator $R_\la=(T-\la)^{-1}$ when $\la\in\rho(T)$.
Theorem \ref{T:B.2} shows that $\bb C\setminus\bb R\subset \rho(T)$.
The set $\rho(T)$ may also intersect the real axis but recall that it is necessarily open.

Suppose $\la\in\rho(T)$ and $[f]\in L^2(w)$.
Since $R_\la$ is an operator defined on all of $L^2(w)$ the class $[f]$ determines the class $[u]=R_\la([f])$ and hence the pair $([u],\la[u]+[f])\in T\subset T_\mx$ uniquely.
The evaluation operator $E$, in turn, determines a unique representative $v$ of $[u]$ satisfying $Jv'+qv=w(\la v+f)$ and $v(x_0)\in N_0^\perp$.
We denote $v$ by $E_\la f$, \ie, $E_\la$ is a map from $L^2(w)$ to $\BVl^\#((a,b))^n$.
Note that we write $E_\la f$ since this does not depend on the representative chosen from $[f]$.

\begin{lem}\label{L:5.5}
$E_\la$ satisfies a resolvent relation, \ie,
$$E_\la-E_\mu=(\la-\mu)E_\la R_\mu$$
whenever $\la$ and $\mu$ are in $\rho(T)$.
Moreover, given $f\in \mc L^2(w)$ and $x\in(a,b)$, the function $\la\mapsto (E_\la f)(x)$ is analytic in $\rho(T)$.
\end{lem}

\begin{proof}
The first claim follows after using the resolvent relation for $R_\la$ (cf. Theorem \ref{T:B.2}) in
$E_\la f-E_\mu f=E(R_\la f-R_\mu f,\la(R_\la f-R_\mu f)+(\la-\mu)R_\mu f)$.

The boundedness of the operators $E$, $R_\la$, and $R_\mu$ and the resolvent relation show that $\la\mapsto (E_\la f)(x)$ is continuous for given $f$ and $x$.
From this analyticity follows since
$$\frac{(E_\la f)(x)-(E_\mu f)(x)}{\la-\mu}=(E_\la(R_\mu f))(x)$$
has a limit as $\la$ tends to $\mu$.
\end{proof}

If $\la\in\rho(T)$ and $x\in(a,b)$ we will show that there is a function $G(x,\cdot,\la)$ satisfying $(E_\la f)(x)=\int G(x,\cdot,\la)wf$ for all $([f],[u])\in R_\la$. The function $G$ is called Green's function for $T$.

\begin{thm}
If $T$ is a self-adjoint restriction of $T_\mx$, then there exists, for given $x\in(a,b)$ and $\la\in\rho(T)$, a matrix $G(x,\cdot,\la)$ such that the columns of $G(x,\cdot,\la)^*$ are in $L^2(w)$ and
$$(E_\la f)(x)=\int G(x,\cdot,\la)wf.$$
\end{thm}

\begin{proof}
For fixed $\la\in\rho(T)$ and $x\in(a,b)$, the components of $f\mapsto (E_\la f)(x)$ are bounded linear functionals on $L^2(w)$.
By Riesz's representation theorem we have $(E_\la f)(x)_j =\<G_j(x,\cdot,\la)^*,f\>$
where $G_j(x,\cdot,\la)^*$ is in $L^2(w)$.
\end{proof}

We close this section by mentioning, that the linear relation $T$ contains a linear operator $T_0$ with the same domain as $T$.
To show this define the closed space $\mc H_\infty=\{[f]\in L^2(w):([0],[f])\in T_\mx\}$ and its orthogonal complement $\mc H_0 =L^2(w)\ominus\mc{H_\infty}$.
According to Theorem \ref{T:ALR.3.1} the relation $T_0=T\cap(\mc H_0\exsum \mc H_0)$ is then a densely defined self-adjoint linear operator.
In fact, $\dom T_0=\dom T$.
$T_0$ is called the operator part of $T$.
By Theorem \ref{T:ALR.3.2} we also have that $\rho(T_0)=\rho(T)$ and that the resolvent of $T_0$ (with respect to $\mc H_0\exsum \mc H_0$) is given by $R_\la\cap(\mc H_0\exsum\mc H_0)$ when $R_\la$ denotes the resolvent of $T$.

\section{The spectral transformation} \label{sst}
\subsection{Properties of Green's function}\label{sst.1}
In this section we fix a point $\la\not\in\Lambda$ and denote, as before, the fundamental matrix for $Ju'+qu=\la wu$ which is the identity at a given point $x_0$ by $U(\cdot,\la)$.
The solutions of $Ju'+qu=\la wu$ form a vector space of dimension $n$ and they may have infinite or finite norm.
Accordingly we define $\mc L_\pm(\la)$ to be the space of solutions of $Ju'+qu=\la wu$ such that $\|\chi_{(x_0,b)}u\|<\infty$ and $\|\chi_{(a,x_0)}u\|<\infty$, respectively.
Note that the space $\mc L_0$ introduced in \eqref{170817.1} is a subset of $\mc L_\pm(\la)$ for any $\la$.
Any solution $u$ of $Ju'+qu=\la wu$ is given by $u=U(\cdot,\la)\eta$ for an appropriate $\eta\in \bb C^n$.
The sets of all $\eta$ for which $U(\cdot,\la)\eta\in \mc L_\pm(\la)$ are denoted by $N_\pm(\la)$
and the associated orthogonal projections from $\bb C^n$ by $P_\pm(\la)$, respectively.

If $T$ is a self-adjoint restriction of $T_\mx$, $\la\in\rho(T)$, and $f\in \mc L^2(w)$ we get from Lemma \ref{varconst2}, the variation of constants formula,
\begin{multline}\label{170812.1}
(E_\la f)(x)=U(x,\la)\big(u_0+J^{-1}\int_{(x_0,x)}U(\cdot,\ol)^*wf\big)\\ +\frac12 U^+(x,\la)J^{-1}U(x,\ol)^*\Delta_w(x)f(x)
\end{multline}
when $x\geq x_0$ and
\begin{multline}\label{170812.2}
(E_\la f)(x)=U(x,\la)\big(u_0-J^{-1}\int_{(x,x_0)}U(\cdot,\ol)^*wf\big)\\ -\frac12 U^-(x,\la)J^{-1}U(x,\ol)^*\Delta_w(x)f(x)
\end{multline}
when $x\leq x_0$.
Of course, we also have
$$(E_\la f)(x)=\int G(x,\cdot,\la)wf$$
and we will obtain properties of $G$ by comparing these identities.
To this end we will determine $u_0$ when the support of $f$ is compact.

Since $E_\la f\in L^2(w)$ we must then have that $u_0+J^{-1}\int_{(x_0,b)}U(\cdot,\ol)^*wf\in N_+(\la)$
and $u_0-J^{-1}\int_{(a,x_0)}U(\cdot,\ol)^*wf\in N_-(\la)$.
The first of these conditions may be written as
\begin{equation}\label{170811.1}
(\id-P_+(\la))u_0=\int(P_+(\la)-\id)J^{-1}\chi_{(x_0,b)}U(\cdot,\ol)^*wf
\end{equation}
while the second is
\begin{equation}\label{170811.2}
(\id-P_-(\la))u_0=\int(\id-P_-(\la))J^{-1}\chi_{(a,x_0)}U(\cdot,\ol)^*wf.
\end{equation}
We also know that $E_\la f$ satisfies $n_+$ boundary conditions, \ie,
$$(g_j^*JE_\la f)^-(b)-(g_j^*JE_\la f)^+(a)=0$$
for $j=1,..., n_+$.
Introduce the matrices
$$A_+(\la)=(g^*J U(\cdot,\la)P_+(\la))^-(b)$$
and
$$A_-(\la)=(g^*J U(\cdot,\la)P_-(\la))^+(a)$$
where $g=(g_1, ..., g_{n_+})$.
Then the boundary conditions can be written as
\begin{equation}\label{170811.3}
(A_+(\la)-A_-(\la))u_0 = \int(-A_-(\la)\chi_{(a,x_0)}-A_+(\la)\chi_{(x_0,b)})J^{-1}U(\cdot,\ol)^*wf.
\end{equation}
Finally recall the space $N_0=\{u(x_0):u\in \mc L_0\}\subset\bb C^n$ and the associated orthogonal projection $P$.
Since $(E_\la f)(x_0)\in N_0^\perp$ we have
\begin{equation}\label{170811.4}
(\id-P)(E_\la f)(x_0)=(\id-P)u_0=0.
\end{equation}
We may collect equations \eqref{170811.1} --- \eqref{170811.4} in a single system
\begin{equation}\label{170816.1}
F(\la)u_0=\int(b_-(\la)\chi_{(a,x_0)}+b_+(\la)\chi_{(x_0,b)})U(\cdot,\ol)^*wf
\end{equation}
for appropriate choices of the matrices $F(\la)$, $b_-(\la)$, and $b_+(\la)$.

\begin{comment}
$$F(\la)=\begin{pmatrix}\id-P_+(\la)\\ \id-P_-(\la)\\ A_+(\la)-A_-(\la)\\ \id-P\end{pmatrix},
 b_-(\la)=\begin{pmatrix}0\\ \id-P_-(\la)\\-A_-(\la)\\ 0\end{pmatrix}J^{-1},
 b_+(\la)=\begin{pmatrix}P_+(\la)-\id\\ 0\\ -A_+(\la)\\ 0\end{pmatrix}J^{-1}$$
\end{comment}

Now suppose $v_0$ is in the kernel of the matrix $F(\la)$ and let $v=U(\cdot,\la)v_0$.
Then $v$ is in $\mc L_+(\la)\cap \mc L_-(\la)$ and hence $(v,\la v)$ is a representative of an element in $D_\la$.
But $v$ also satisfies the boundary conditions, \ie, $([v],\la [v])$ is in $T\subset T^*$.
Since $\la$ is not an eigenvalue, it follows that $v\in \mc L_0$ and $v(x_0)=v_0\in N_0\cap N_0^\perp$.
Hence $v_0=0$ so that $F(\la)$ has full rank $n$ and $F^\dag(\la)=(F(\la)^*F(\la))^{-1}F(\la)^*$ is a left inverse of $F(\la)$.
Since the existence of a solution of the system \eqref{170816.1} is not in question its right-hand side is in the range of $F(\la)$.
Thus we may apply $F^\dag$ to both sides of \eqref{170816.1} to obtain
$$u_0=\int (H_-(\la)\chi_{(a,x_0)}+H_+(\la)\chi_{(x_0,b)})U(\cdot,\ol)^*wf$$
where
$$H_-(\la)=PF^\dag(\la)b_-(\la)\text{ and } H_+(\la)=PF^\dag(\la)b_+(\la).$$
Note that we were allowed to multiply with $P$ from the left, since $u_0\in N_0^\perp$.

Now introduce the functions $M_\pm(\la)=H_\pm(\la)\pm\frac12 J^{-1}$ and recall that $U^\pm(x,\la)=U(x,\la)\pm\frac12(U^+(x,\la)-U^-(x,\la))$.
Then we may rewrite equations \eqref{170812.1} and \eqref{170812.2} as
\begin{equation}\label{180108.1}
(E_\la f)(x)=\int U(x,\la) H(x,\cdot,\la) U(\cdot,\ol)^* wf
\end{equation}
where
$$H(x,\cdot,\la)=M_-(\la)\chi^\#_{(a,x_0)}+M_+(\la)\chi^\#_{(x_0,b)}-\frac12 J^{-1}\sgn(\cdot-x)+S(x,\la)\chi_{\{x\}}$$
and $S(x,\la)=\frac14U(x,\la)^{-1}(U^+(x,\la)-U^-(x,\la))J^{-1}$.
Note that $S(x,\la)=0$ at points $x$ where $Q$ and $W$ are continuous, \ie, almost everywhere.

We illustrate the above by the following simple example: $n=1$, $(a,b)=\bb R$, $J=-i$, $q=0$, $w=1$, and $x_0=0$.
Then $U(x,\la)=\e^{i\la x}$.
Assume first that $\Im(\la)>0$.
Then we have $N_+(\la)=\bb C$ and $N_-(\la)=\{0\}$.
We have no boundary conditions ($n_+=n_-=0$) and $\mc L_0=\{0\}$.
This gives $F(\la)=(0,1,0)^\top $, $b_+=(0,0,0)^\top$ and $b_-=(0,i,0)^\top$.
Since $F^\dag(\la)=(0,1,0)$ we find $H_+(\la)=0$ and $H_-(\la)=i$ and thus $M_+(\la)=M_-(\la)=i/2$.
Thus $H(x,y,\la)$ equals zero above the line $y=x$ and $i$ below it.
When $\Im(\la)<0$ we get instead $H_+(\la)=-i$, $H_-(\la)=0$ and $M_+(\la)=M_-(\la)=-i/2$.
Now $H(x,y,\la)=-i$ above the line $y=x$ and $H(x,y,\la)=0$ below it.
Thus, for $\Im\la\neq0$,
$$G(x,y,\la)=\frac{i}2\frac{\Im\la}{|\Im\la|}\e^{i\la(x-y)}\begin{cases}2&\text{if $(x-y)\Im\la>0$}\\ 1&\text{if $x=y$}\\ 0&\text{if $(x-y)\Im\la<0$.}\end{cases}$$

\subsection{The \texorpdfstring{$M$}{M}-function}\label{sst.2}
For this and the next section we need somewhat stronger assumptions than before and shall, in addition to Hypothesis \ref{H:4.1}, also require the validity of the following premise.
\begin{hyp}\label{H:6.1}
$\Lambda$ is a closed set which does not intersect the real line and contains only isolated points.
\end{hyp}

We proved in Section \ref{sde.3} that $u(x,\cdot)$ is analytic in $\bb C\setminus\Lambda$ when $u(\cdot,\la)$ is a solution of the initial value problem $Ju'+qu=\la wu$, $u(x_0)=u_0\in\bb C^n$.
Thus, if $U(\cdot,\la)$ is a fundamental matrix for $Ju'+qu=\la wu$ such that $U(x_0,\la)=\id$, then $U(x,\cdot)$ is analytic on $\bb{C}\setminus\Lambda$.
We choose, as always, $x_0$ to be a point of continuity of $Q$ and $W$, the antiderivatives of $q$ and $w$.

Our goal is to define a matrix-valued Herglotz-Nevanlinna function which encodes the spectral information of $T$, a self-adjoint restriction of $T_\mx$.
Introduce the spaces $B_\pm$ collecting respectively the vectors $\int_{(x_0,b)}U(\cdot,\ol)^*wf$ and $\int_{(a,x_0)}U(\cdot,\ol)^*wf$ when $f$ runs through the compactly supported functions in $L^2(w)$.
By Lemma~\ref{L:5.1} the span of $B_-\cup B_+$ equals $N_0^\perp$ and, mimicking its proof, we also see that $B_\pm$ are independent of $\la$.
We denote the orthogonal projections onto the spaces $B_\pm$ by $\Omega_\pm$.

\begin{lem}
If $\Omega_0$ is the orthogonal projection onto $B_+\cap B_-$, then $PM_-(\la)\Omega_0=PM_+(\la)\Omega_0$.
\end{lem}

\begin{proof}
Lemma~\ref{T4.1.4a} gives $U^\pm(\cdot,\la)J^{-1}U^\pm(\cdot,\ol)^*=J^{-1}$ so that $U(x,\la)(S(x,\la)-S(x,\ol)^*)U(x,\ol)^*=0$.
This implies
$$0=\<g,E_\la f\>-\<E_\ol g,f\>=\int\int g(x)^*w(x) U(x,\la)K(x,y,\la)U(y,\ol)^*w(y)f(y)$$
where $K(x,y,\la)$ equals
$$M_-(\la)\chi^\#_{(a,x_0)}(y)+M_+(\la)\chi^\#_{(x_0,b)}(y)-M_-(\ol)^*\chi^\#_{(a,x_0)}(x)-M_+(\ol)^*\chi^\#_{(x_0,b)}(x).$$
By picking $f$ and $g$ both to be supported on $(a,x_0]$ we may prove
\begin{equation}\label{171230.1}
\alpha^*(M_-(\la)-M_-(\ol)^*)\alpha'=0
\end{equation}
for all $\alpha,\alpha'\in B_-$.
Similarly, choosing supports of both $f$ and $g$ in $[x_0,b)$ we find
\begin{equation}\label{171230.2}
\beta^*(M_+(\la)-M_+(\ol)^*)\beta'=0
\end{equation}
for all $\beta,\beta'\in B_+$.
We also get
\begin{equation}\label{171230.3}
\alpha^*(M_+(\la)-M_-(\ol)^*)\beta=0
\end{equation}
and
\begin{equation}\label{171230.4}
\beta^*(M_-(\la)-M_+(\ol)^*)\alpha=0
\end{equation}
whenever $\alpha\in B_-$ and $\beta\in B_+$.

Now suppose $\zeta\in B_-\cap B_+$.
Then the identities \eqref{171230.1} and \eqref{171230.3} show that $\alpha^*(M_+(\la)-M_-(\la))\zeta=0$ for all $\alpha\in B_-$.
The other two give $\beta^*(M_+(\la)-M_-(\la))\zeta=0$ for all $\beta\in B_+$.
\end{proof}

We may now make the following definition when $\la$ is in $\rho(T)\setminus\Lambda$.
$$M(\la)\zeta=\begin{cases}
  PM_+(\la)\zeta&\text{if $\zeta\in B_+$,}\\   PM_-(\la)\zeta&\text{if $\zeta\in B_-$, and}\\   0&\text{if $\zeta\in N_0$.}\end{cases}$$
Of course, elsewhere $M(\la)$ is defined by linearity.
Since we may replace the term $U(\cdot,\ol)^*w$ in equation \eqref{180108.1} by $PU(\cdot,\ol)^*w$ and since $M_\pm=PM_\pm\pm\frac12(\id-P)J^{-1}$, we may replace $H$ in that equation by $\tilde H=HP$ and still retain the fact that it yields Green's function.
Thus, setting
\begin{multline*}
\tilde H(x,y,\la)=M(\la)+\frac12(\id-P)J^{-1}P \sgn(y-x_0)\\
 -\frac12 J^{-1}P\sgn(y-x)+S(x,\la)P\chi_{\{x\}}(y)
\end{multline*}
we still get
\begin{equation}\label{180108.2}
(E_\la f)(x)=\int U(x,\la) \tilde H(x,\cdot,\la) U(\cdot,\ol)^* wf.
\end{equation}

\begin{lem}\label{T:6.2}
The function $M$ extends to all of $\rho(T)$ and has the following properties:
(1) $M(\la)=M(\ol)^*$.
(2) $(M(\la)-M(\la)^*)/(2i \Im(\la))\geq0$ if $\Im(\la)\neq0$.
(3)~$M$ is analytic.
\end{lem}

\begin{proof}
We first establish the above statements in $\rho(T)\setminus\Lambda$.

For $\xi\in\bb C^n$ and $\zeta\in B_-$ we have $\xi^*M(\la)\zeta=\xi^*PM_-(\la)\zeta=(P\xi)^*M_-(\la)\zeta$.
But $P\xi=\alpha+\beta$ with $\alpha\in B_-$ and $\beta\in B_+$.
Thus, on account of the identities \eqref{171230.1} and \eqref{171230.4} and since $\zeta=P\zeta$, we get
\begin{multline*}
\xi^*M(\la)\zeta=\alpha^*M_-(\ol)^*\zeta+\beta^*M_+(\ol)^*\zeta=(PM_-(\ol)\alpha)^*\zeta+(P M_+(\ol)\beta)^*\zeta\\
 =(\alpha+\beta)^*M(\ol)^*\zeta=\xi^* (M(\ol)P)^*\zeta.
\end{multline*}
Now use that $MP=M$.
A similar proof works when $\zeta\in B_+$.
If $\zeta\in N_0$ we have $\xi^*M(\la)\zeta=0$ and $\xi^*M(\ol)^*\zeta=(M(\ol)\xi)^*\zeta$.
The latter is also zero since $M(\ol)\xi\in N_0^\perp$.
This establishes (1) for $\la\in\rho(T)\setminus\Lambda$.

Suppose $\supp f\in[c,d]\subset(a,b)$ and denote by $\xi_\pm$ the integrals of $U(\cdot,\ol)^*wf$ over the intervals $[x_0,b)$ and $(a,x_0]$, respectively. For $x$ outside $[c,d]$ we get
$$(E_\la f)(x)=U(x,\la) (M(\la)\mp\frac12J^{-1}P)(\xi_++\xi_-) +\frac12 U(x,\la)(\id-P)J^{-1}P (\xi_+-\xi_-)$$
when we use the upper sign for $x<c$ and the lower sign for $x>d$.
The columns of $U(x,\la)(\id-P)$ are in $\mc L_0\subset\mc L^2(w)$ and thus satisfy the boundary conditions defining $T$.
Since $E_\la f\in \mc L^2(w)$ also satisfies these boundary conditions it follows that the function
$$h_1(x,\la)=U(x,\la)(M(\la)+\frac12J^{-1}P\sgn(x-x_0))$$
is in $\mc L^2(w)$ and satisfies the boundary conditions.
Now let $h=h_1(\cdot,\ol)-h_1(\cdot,\ov\mu)$.
Since $h$ is continuous at $x_0$ it satisfies the differential equation $Jh'+(q-\ov\mu w)h=(\ol-\ov\mu)wh_1(\cdot,\ol)$ on all of $(a,b)$.
From the above we know that $h\in L^2(w)$ satisfies the boundary conditions for $T$.
Thus $(h,\ov\mu h+(\ol-\ov\mu)h_1(\cdot,\ol))\in \mc T$.
This is the same as to say that $h$ is a representative of $(\ol-\ov\mu)R_{\ov\mu}h_1(\cdot,\ol)$.
Since $h(x_0)=M(\ol)-M(\ov\mu)\in N_0^\perp$ we even have
$$h=h_1(\cdot,\ol)-h_1(\cdot,\ov\mu)=(\ol-\ov\mu)E_{\ov\mu}h_1(\cdot,\ol).$$
Taking the adjoint and evaluating at $x_0$ gives
\begin{equation}\label{180103.1}
M(\la)-M(\mu)=(\la-\mu)(E_{\ov\mu}h_1(\cdot,\ol))(x_0)^*.
\end{equation}

Now we are in a position to prove (2) away from $\Lambda$.
We simply note that, by (1), $M(\la)^*=M(\ol)$ so that
$$M(\la)-M(\la)^*=(\la-\ol)(E_{\la}h_1(\cdot,\ol))(x_0)^*.$$
However, evaluating equation \eqref{180108.2} at $x_0$ gives $(E_\la f)(x_0)=\int h_1(\cdot,\ol)^*wf$ and, in particular,
$$(E_{\la}h_1(\cdot,\ol))(x_0)=\int h_1(\cdot,\ol)^*wh_1(\cdot,\ol)\geq0.$$

To prove (3) outside $\Lambda$ suppose that $\mu\in\rho(T)\setminus\Lambda$ and that $\la$ is in a ball around $\mu$ which still lies in $\rho(T)\setminus\Lambda$.
Equation \eqref{180103.1} shows then that $M$ is analytic at $\mu$.

Finally, using Hypothesis \ref{H:6.1} and the next Lemma \ref{L:6.3}, we may extend $M$ analytically to all of $\rho(T)$.
Properties (1) and (2) remain intact by continuity.
\end{proof}

\begin{lem}\label{L:6.3}
Let $\la_0$ be a non-real isolated singularity of a function $m$ which is analytic in a punctured disk about $\la_0$.
If $\Im m(\la)\geq0$ on its domain of definition, then $\la_0$ is a removable singularity.
\end{lem}

\begin{proof}
If $\la_0$ is a pole or an essential singularity of $m$, then there must be points in a neighborhood of $\la_0$ such that $\Im(\la)<0$.
\end{proof}

Lemma \ref{T:6.2} guarantees that $M$ is a matrix-valued Herglotz-Nevanlinna function.
It is well-known, that any such function has the representation
$$M(\la)=A\la +B+\int \big(\frac1{t-\la}-\frac{t}{t^2+1}\big)dN(t)$$
where $A$ is a non-negative matrix, $B$ a Hermitian matrix and $t\mapsto N(t)$ a non-decreasing, left-continuous, matrix-valued function.
We will denote the derivative of $N$, a non-negative distribution or measure by $\nu$.
Recall that the entries of $\nu$ are absolutely continuous with respect to $\tr\nu$ and let $\tilde N$ be the Radon-Nikodym derivative of $\nu$ with respect to $\tr\nu$.
Of course, associated with the measure $\nu$ is the Hilbert space $L^2(\nu)$.
We shall denote the inner product and norm in $L^2(\nu)$ by $\<\cdot,\cdot\>_\nu$ and $\|\cdot\|_\nu$, respectively.
For emphasis we subsequently append a subscript $w$ for norms and inner products in $L^2(w)$ when appropriate.

\subsection{The spectral transformation}\label{sst.3}
Throughout this section we require the validity of both Hypotheses \ref{H:4.1} and \ref{H:6.1}.
In the course of action we follow the proof of Theorem 15.5 in Bennewitz \cite{Bennewitz-sths} almost verbatim.
Extra attention is needed in the proof of Lemma \ref{L:6.8} to tend to differences stemming from the absence of the definiteness condition.
We include the other details since \cite{Bennewitz-sths} is perhaps not readily available.

If $f\in L^2(w)$ has compact support, we define for any $\la\in\bb C\setminus\Lambda$
$$(\mc Ff)(\la)=\int U(\cdot,\ol)^*wf.$$
$\mc Ff$ is called the spectral transform or Fourier transform of $f$.
Note that $\la\mapsto(\mc Ff)(\la)$ and $\la\mapsto(\mc Ff)(\ol)^*$ are analytic in $\bb C\setminus\Lambda$.
According to Lemma \ref{L:5.1} the vectors $(\mc F f)(\la)$ always lie in $N_0^\perp=\ran P$.
We will mostly be concerned with the restrictions of transforms to the real line.
We will still use the notation $\mc F f$ for such a restriction, since the meaning will be clear from the context.

The following lemma indicates that all spectral information of the self-adjoint relation $T$ is encoded in the function $M$ since it is encoded in the resolvent of $T$.
\begin{lem}\label{L:6.5}
Suppose $f,g\in L^2(w)$ are compactly supported.
Then the function $\la\mapsto \<g,E_\la f\>-(\mc Fg)(\ol)^* M(\la)(\mc Ff)(\la)$ is analytic in $\bb C\setminus\Lambda$.
\end{lem}

\begin{proof}
In view of the representation \eqref{180108.2} for $E_\la f$, the claim follows simply from the fact that $\tilde H(x,y,\la)-M(\la)=\frac12(\id-P)J^{-1}P \sgn(y-x_0)-\frac12 J^{-1}P\sgn(y-x)+S(x,\la)\chi_{\{x\}}(y)$  is analytic in $\bb C\setminus\Lambda$.
\end{proof}

We now recall from Appendix \ref{ALR.4} the spectral projections $\pi(B)$ associated with the self-adjoint operator $T_0$ and, by extension, with the self-adjoint relation $T$.
In particular, remember that $\Pim_{f,g}$ is the anti-derivative of the distribution associated with the complex measure $B\mapsto \<f,\pim(B)g\>$
and that $\pim(\bb R)$ is the orthogonal projection from $L^2(w)$ onto $\mc H_0$.
We will denote $\pim(\bb R)$ by $\bb P$.

\begin{lem}\label{L:6.5a}
Suppose $f\in L^2(w)$ is compactly supported and $c$ and $d$ are points of continuity of both $\Pim_{f,f}$ and $N$.
Let $\Gamma$ be the contour described by the rectangle with vertices $c\pm i\varepsilon$ and $d\pm i\varepsilon$ where $\varepsilon>0$ is chosen so that the rectangle and its interior lie entirely in $\bb C\setminus\Lambda$ (recall that $\bb R\cap\Lambda$ is empty).
Then
\begin{equation}\label{180125.1}
\oint_\Gamma(\mc Ff)(\ol)^* M(\la)(\mc Ff)(\la)d\la=-2\pi i \int_{(c,d)} (\mc Ff)^*\nu(\mc Ff)
\end{equation}
and
\begin{equation}\label{180125.2}
\oint_\Gamma \<f,R_\la f\>d\la=-2\pi i \int_{(c,d)} d\Pim_{f,f}.
\end{equation}
\end{lem}

\begin{proof}
Using the Nevanlinna representation for $M$, exchanging the order of integration, and employing Cauchy's integral formula we obtain for the right-hand side of equation \eqref{180125.1}
$$\int \oint_\Gamma (\mc Ff)(\ol)^* \big(\frac1{t-\la}-\frac{t}{t^2+1}\big)\tilde N(t)(\mc Ff)(\la)d\la\tr\nu(t)=-2\pi i \int_{(c,d)} (\mc Ff)^*\nu(\mc Ff).$$
Exchanging the order of integration is permissible, if the integral is absolutely convergent.
A difficulty arises when $\la=t$ leading to integrals of the type
$$\int_{-1}^{1}\int_{(-1,1)} \frac{\sigma'(t)}{\sqrt{t^2+s^2}} ds$$
when $\sigma$ is non-decreasing on $(-1,1)$ and differentiable at $0$.
But integration by parts shows that such integrals are finite.
This proves equation \eqref{180125.1}.
The proof of equation \eqref{180125.2} is similar.
\end{proof}

\begin{lem}\label{L:6.6}
When $f\in L^2(w)$ is compactly supported, then $\mc Ff$ is in $L^2(\nu)$.
The map $\mc F$ extends, by continuity, to all of $L^2(w)$.
Moreover,
$$\Pim_{f,g}(t)=\<f,\pim((-\infty,t))g\>_w=\int_{(-\infty,t)} (\mc F f)^* \nu (\mc Fg)$$
whenever $f,g\in L^2(w)$.
In particular, $\<f,\bb Pg\>_w=\<\mc F f,\mc F g\>_\nu$ and $\ker\mc F=\mc H_\infty$.
\end{lem}

\begin{proof}
Suppose $f\in L^2(w)$ is compactly supported.
Then Lemma \ref{L:6.5} and Cauchy's theorem show that
$$\oint_\Gamma \<f,R_\la f\>=\oint_\Gamma(\mc Ff)(\ol)^* M(\la)(\mc Ff)(\la)$$
since the integral on the left exists according to Lemma \ref{L:6.5a}.
That lemma gives also
\begin{equation}\label{180212.1}
\int_{[c,d)} d\Pim_{f,f}=\int_{[c,d)} (\mc Ff)^*\nu(\mc Ff)
\end{equation}
when $c$ and $d$ are points of continuity for $\Pim_{f,f}$ and $N$.
Since $\Pim_{f,f}$ and $N$ are both left-continuous equation \eqref{180212.1} actually holds always.
Letting $[c,d)$ tend to $\bb R$ proves that $\mc Ff\in L^2(\nu)$.
Furthermore, polarization gives
$$\int_{[c,d)} d\Pim_{f,g}=\int_{[c,d)} (\mc Ff)^*\nu(\mc Fg)$$
when $g\in L^2(w)$ is also compactly supported.

Now let $f$ be an arbitrary element of $L^2(w)$ and $n\mapsto[a_n,b_n]$ be a sequence of intervals in $(a,b)$ converging to $(a,b)$.
Setting $f_n=f\chi_{[a_n,b_n]}$ we see that
$$\|\mc F f_n-\mc Ff_m\|_\nu=\|\bb P(f_n-f_m)\|_w\leq \|f_n-f_m\|_w$$
showing that $n\mapsto \mc Ff_n$ is a Cauchy sequence in $L^2(\nu)$ and thus convergent.
By interweaving sequences it follows that the limit of this Cauchy sequence does not depend on how $f$ is approximated.
We denote the limit by $\mc F f$ thereby extending our definition of the Fourier transform to all of $L^2(w)$.
\end{proof}

Next we define a transform $\mc G:L^2(\nu)\to L^2(w)$.
Again, we will begin by considering compactly supported elements in $L^2(\nu)$.
Specifically, if $\hat f$ is such an element, we define
$$(\mc G \hat f)(x)=\int U(x,\cdot)\nu \hat f.$$

\begin{lem}\label{L:6.6a}
When $\hat f\in L^2(\nu)$ is compactly supported, then $\mc G\hat f$ is in $L^2(w)$.
The map $\mc G$ extends, by continuity, to all of $L^2(\nu)$.
We have that
\begin{equation}\label{180110.1}
\<g,\mc G\hat f\>_w=\<\mc Fg,\hat f\>_\nu
\end{equation}
for all $g\in L^2(w)$ and all $\hat f\in L^2(\nu)$.
Moreover, $\mc G\circ\mc F=\bb P$ and $(\ran\mc F)^\perp=\ker\mc G$.
\end{lem}

\begin{proof}
Suppose that $\hat f\in L^2(\nu)$ is compactly supported, denote $\mc G\hat f$ by $f$, and let $f_n=f\chi_{[a_n,b_n]}$.
Upon changing the order of integration we get
$$\|f_n\|_w^2=\<f_n,f\>_w=\<\mc Ff_n,\hat f\>_\nu\leq \|\mc Ff_n\|_\nu \|\hat f\|_\nu.$$
Lemma \ref{L:6.6} implies $\|\mc Ff_n\|_\nu=\|\bb Pf_n\|_w\leq \|f_n\|_w$ and hence we have $\|f_n\|_w\leq \|\hat f\|_\nu$.
This is the case for every interval $[a_n,b_n]\subset (a,b)$ so it follows that $\mc G\hat f\in L^2(w)$.

As before we extend the domain of definition of $\mc G$ from the compactly supported functions in $L^2(\nu)$ to all of $L^2(\nu)$.
Specifically, for a general element $\hat f$ in $L^2(\nu)$ set $\hat f_n=\hat f\chi_{[-n,n]}$.
Then, according to what we just proved, $\|\mc G\hat f_n-\mc G\hat f_m\|_w\leq \|\hat f_n-\hat f_m\|_\nu$ which implies that $n\mapsto \mc G\hat f_n$ is a Cauchy sequence in $L^2(w)$ and thus convergent.
We denote the limit by $\mc G\hat f$.

Now suppose $g\in L^2(w)$, $\hat f\in L^2(\nu)$, $[a_k,b_k]\subset(a,b)$ and $[-n,n]\subset\bb R$.
Upon changing the order of integration we get (as before)
$$\<g\chi_{[a_k,b_k]},\mc G (\hat f\chi_{[-n,n]})\>_w=\<\mc F(g\chi_{[a_k,b_k]}),\hat f\chi_{[-n,n]}\>_\nu.$$
Let $[a_k,b_k]\times[-n,n]$ approach $(a,b)\times\bb R$ to obtain equation \eqref{180110.1}.

If $\hat f=\mc F f$ for some $f\in L^2(w)$ we get from equation \eqref{180110.1} and Lemma \ref{L:6.6} that $\<g,(\mc G\circ\mc F)f\>=\<g,\bb Pf\>$ which shows that $\mc G\circ\mc F=\bb P$.
Also the last claim follows immediately from equation \eqref{180110.1}.
\end{proof}

\begin{lem}\label{L:6.7}
If $\Im(\la)\neq0$, then $\mc F(R_\la g)(t)=(\mc Fg)(t)/(t-\la)$.
\end{lem}

\begin{proof}
First note that $t\mapsto \hat g(t)/(t-\la)$ is in $L^2(\nu)$ if $\hat g$ is.
From Theorem \ref{T:B.8} and Lemma \ref{L:6.6} we get
$$\<f,R_\la g\>_w=\int \frac1{t-\la}\ d\Pim_{f,g}(t)=\int (\mc F f)(t)^* \nu(t) \frac{(\mc Fg)(t)}{t-\la}=\<\mc Ff,(\mc Fg)/(\cdot-\la)\>_\nu.$$
In particular, $\|R_\la g\|_w^2=\<\mc F(R_\la g),(\mc Fg)/(\cdot-\la)\>_\nu$.
On the other hand
$$\|R_\la g\|_w^2=\<g,R_\ol R_\la g\>_w=\frac1{\la-\ol}\<g,(R_\la- R_\ol)g\>_w=\|\mc Fg/(\cdot-\la)\|^2_\nu.$$
Lemma \ref{L:6.6} also implies that $\|R_\la g\|_w^2=\|\mc F(R_\la g)\|_\nu^2$.
Thus the four terms appearing in the expansion of $\|\mc F(R_\la g)-\mc Fg/(\cdot-\la)\|^2$ cancel each other leaving $0$.
\end{proof}

\begin{lem}\label{L:6.8}
The kernel of $\mc G$ is trivial, $\mc F\circ\mc G=\id$, $\mc F$ is surjective, and $\mc F^*=\mc G$.
In particular, $\mc F|_{\mc H_0}:\mc H_0\to L^2(\nu)$ is unitary.
\end{lem}

\begin{proof}
Assume first $\ker \mc G=\{0\}$ so that $\ran\mc F$ is dense.
Lemmas \ref{L:6.6} and \ref{L:6.6a} imply
$$\<(\mc F\circ\mc G-\id)\hat g,\mc F f\>_\nu=\<\mc G\hat g,\bb P f\>_w-\<\hat g,\mc Ff\>_w=\<\hat g,\mc F(\bb Pf-f)\>_w=0.$$
Thus $\mc F\circ\mc G=\id$ which implies the other statements claimed.

Now we show that indeed $\ker \mc G=\{0\}$.
Thus suppose $\hat f\in\ker\mc G$ and $\Im(\la)\neq0$.
Then, employing Lemma \ref{L:6.7},
$$0=\<\mc F R_\ol g,\hat f\>_\nu=\int \frac{(\mc Fg)(t)^*\nu(t)\hat f(t)}{t-\la}$$
whenever $g\in L^2(w)$.
Since the Stieltjes transform of a measure is zero only for the zero measure we have $(\mc Fg)^*\nu \hat f=0$ and hence $\<\mc Fg,\hat f\chi_K\>_\nu=\int_K(\mc Fg)^*\nu \hat f=0$ whenever $K$ is a compact subset of $\bb R$.
Lemma \ref{L:6.6a} gives us then $\hat f\chi_K\in\ker\mc G$.
Now apply $\mc G$ to see that $x\mapsto\int_K U(x,t)\nu(t)\hat f(t)$ is the zero element in $L^2(w)$.
Hence, whenever $[c,d]\in(a,b)$,
$$0=\int_{[c,d]}U(x,s)^*w(x)\int_K U(x,t)\nu(t)\hat f(t)=\int_K B(s,t) \nu(t) \hat f(t)$$
where
$$B(s,t)=\int_{[c,d]}U(x,s)^*w(x)U(x,t)=PB(s,t)P.$$
Now fix an arbitrary $s\in\bb R$ and note that, for a sufficiently large interval $[c,d]$, we have that $\alpha^*B(s,s)\alpha=0$ if and only if $\alpha\in N_0$ and hence $\ker B(s,s)=N_0$.
As long as $t$ is sufficiently close to $s$ we also have $\ker B(s,t)=\ker(PB(s,t)P)=N_0$ implying that $\ran B(s,t)^*=N_0^\perp$.
Since $\hat f$ is an element of $L^2(\nu)$ we may assume that its values are in $\ran P=N_0^\perp$.
Thus, for each $t$ sufficiently close to $s$ we may find a vector $\xi(t)$ such that $B(s,t)^*\xi(t)=\hat f(t)$.
Since $B(s,\cdot)\nu\hat f$ is the zero measure, it follows that $\hat f^*\nu\hat f=\xi^*B(s,\cdot)\nu\hat f$ is also the zero measure.
But this means that $\hat f$ is zero almost everywhere with respect to $\nu$.
\end{proof}

It is easy to see that the Fourier transform diagonalizes the relation $T$.
Indeed, suppose $(u,f)\in T$ and hence that $(f-\la u,u)\in R_\la$.
Then Lemma~\ref{L:6.7} gives $(\mc Fu)(t)=(\mc F(f-\la u))(t)/(t-\la)$ which simplifies to $t(\mc Fu)(t)=(\mc Ff)(t)$.
Conversely, if the functions $t\mapsto \hat u(t)$ and $t\mapsto \hat f(t)=t \hat u(t)$ are in $L^2(\nu)$, then $(\mc G\hat u,\mc G\hat f)\in T_0$.

We have proved the following theorem.
\begin{thm} \label{t:main}
Suppose $T$ is a self-adjoint restriction of a relation $T_\mx$ whose coefficients $q$ and $w$ satisfy Hypotheses \ref{H:4.1} and \ref{H:6.1}.
Let $\nu$ be the measure generated by the associated $M$-function.
Then the following statements hold.
\begin{enumerate}
\item There is a continuous map $\mc F:L^2(w)\to L^2(\nu)$ which assigns to a compactly supported element $f\in L^2(w)$ the function defined by $(\mc Ff)(t)=\int U(\cdot,t)^*wf$.
The kernel of $\mc F$ is the space $\mc H_\infty=\{f\in L^2(w):(0,f)\in T\}$.
\item  There is a continuous map $\mc G:L^2(\nu)\to L^2(w)$ which assigns to a compactly supported element $\hat f\in L^2(\nu)$ the function defined by $(\mc G\hat f)(x)=\int U(x,\cdot)\nu \hat f$.
The range of $\mc G$ is the space $\mc H_0=L^2(w)\ominus\mc H_\infty$.
\item The restriction of $\mc F$ to $\mc H_0$ is a unitary operator and $\mc G$ is the inverse of this operator.
\item If $(u,f)\in T$ then $(\mc Ff)(t)=t (\mc Fu)(t)$.
Conversely, if $t\mapsto \hat u(t)$ and $t\mapsto \hat f(t)=t\hat u(t)$ are both in $L^2(\nu)$, then $(\mc G\hat u, \mc G\hat f)\in T_0\subset T$.
\end{enumerate}
\end{thm}

\section{Special case: regular endpoints}\label{sre}
The endpoint $a$ is called \textit{regular}\index{regular endpoint}, if the antiderivatives $Q$ and $W$ of $q$ and $w$ are of bounded variation on $(a,c)$ for some (and hence all) $c\in\iOmega$.
Similarly, $b$ is called regular if $Q$ and $W$ are of bounded variation on $(c,b)$.
Here we allow $a=-\infty$ and $b=\infty$.
An endpoint which is not regular is called \textit{singular}\index{singular endpoint}.

If $a$ is regular, then the measure associated with $\tr w$ is a finite measure on $(a,c)$.
This implies that the components of $wf$ are associated with finite measures on $(a,c)$ whenever $f\in \mc L^2(w)$.
Consequently, in view of Remark \ref{R:2.4}, any balanced solution $u$ of $Ju'+qu=wf$ is of bounded variation on $(a,c)$.
Hence $u^\pm$ and $u^\#$ all have the same limit at $a$ which we denote by the $u(a)$.
In this setup we do not allow (or, if you will, ignore) problems where $a$ carries mass.
Unless $a=-\infty$ this is not actually a restriction, since we may set $w$ and $q$ equal to zero on $(-\infty,a)$ and consider the corresponding problem posed on $(-\infty,b)$.
Similar considerations hold, of course, with the roles of $a$ and $b$ reversed.
$u(a)$ and $u(b)$ are called the boundary values of $u\in\dom\mc T_\mx$ when $a$ and $b$ are regular.

In the following we will consider the case when both $a$ and $b$ are regular.
We emphasize that $\Lambda$ is then a closed set consisting of isolated points so that a part of Hypothesis \ref{H:6.1} is automatically satisfied when both endpoints are regular.
We will not require that $\Lambda\cap\bb R$ is empty unless explicitly mentioned.

\subsection{Boundary conditions}
We begin by showing that the symmetric restrictions of $T_\mx$ are given by linear homogenous conditions on $u(a)$ and $u(b)$.
Recall that the deficiency indices $n_\pm$ of $T_\mn$ are the dimensions of $D_{\pm i}=\{([u],\pm i [u])\in T_\mx\}$.

\begin{lem}\label{L5.1}
If both endpoints of $\iOmega$ are regular, then the minimal relation $T_\mn$ has equal deficiency indices.
\end{lem}
\begin{proof}
Suppose $\la\not\in\Lambda$.
By Corollary \ref{C4.1.4} the space of solutions of $Ju'+qu=\la wu$ has dimension $n$.
Each of these solutions is bounded and, since $\tr w$ is a finite measure, they are all in $\mc L^2(w)$.
Some of them may be representatives of $[0]$ so that $n_\pm\leq n$.
If there is a solution of $Ju'+qu=\la wu$ with $\|u\|=0$, then $wu=0$ so that $u\in \mc L_0$ and also $Ju'+qu=-\la wu$ .
Hence $n_\pm=n-\dim \mc L_0$.
\end{proof}

We may not be able to associate unique boundary values with a given element $[u]\in\dom(T_\mx)$ since different representatives of $[u]$ may have different boundary values.
Worse, even for a fixed $([u],[f])\in T_\mx$ two different representatives of $[u]$ may have different boundary values.
To remedy this situation we make use of the evaluation operator $E$ from Definition \ref{D:5.2} and introduce the map $\mc B$ which assigns to $([u],[f])\in T_\mx$ the boundary values of $E([u],[f])$, \ie, $\mc{B}([u],[f])=\sm{E([u],[f])(a)\\ E([u],[f])(b)}$.
We also define the spaces
$$N=\Big\{\sm{v(a)\\ v(b)}\in\bb C^{2n}: v\in \mc L_0\Big\}$$
and $W=\{\mc B([v],[g]): ([v],[g])\in D_i\oplus D_{-i}\}$.

\begin{lem}\label{L5.2}
Suppose $(u,f)\in\mc T_\mx$.
Then $\sm{u(a)\\ u(b)}\in N$ if and only if $([u],[f])\in\ov{T_\mn}$.
Moreover, $W\cap N=\{0\}$ and $\dim W=\dim(D_i\oplus D_{-i})=2n-2\dim\mc L_0$.
\end{lem}

\begin{proof}
Suppose $\sm{u(a)\\ u(b)}\in N$ and that $v\in\mc L_0$ has the same boundary values as $u$ so that $(u-v)(a)=(u-v)(b)=0$ and $J(u-v)'+q(u-v)=wf$.
Then, by Lagrange's identity \eqref{Lagrange}, we obtain $\<u-v,h\>=\<f,r\>$ for all $(r,h)\in\mc T_\mx$.
Hence $([u-v],[f])=([u],[f])\in T_\mx^*=\ov{T_\mn}$ and this proves one direction of the first claim.

Conversely, suppose $([u],[f])\in\ov{T_\mn}$ and $([v],[g])\in T_\mx=\ov{T_\mn}^*$.
Then, by Lagrange's identity \eqref{Lagrange} we have $v(b)^*Ju(b)-v(a)^*Ju(a)=0$, \ie,  $\sm{u(a)\\ u(b)}$ is in the kernel of the linear functional given by the row $(-v(a)^*J,v(b)^*J)$.
For arbitrary $v_0\in\bb C^n$ and $g\in L^2(w)$ we have a solution of the initial value problem $Jv'+qv=wg$, $v(a)=v_0$.
The variation of constants formula (Lemma \ref{varconst2}) gives
$$v(b)=U(b,0)\big(v(a)+J^{-1}\int_{(a,b)} U(\cdot,0)^*wg\big).$$
Denoting $\dim \mc L_0$ by $k$, Lemma \ref{L:5.1} shows that the vectors $\int_{(a,b)} U(\cdot,0)^*wg$ span a space of dimension $n-k$.
Thus we have $2n-k$ linearly independent functionals whose kernels contain $\sm{u(a)\\ u(b)}$, \ie,
$\sm{u(a)\\ u(b)}$ is in the kernel $K$ of a matrix of rank $2n-k$ so that $\dim K=k$.
Now note that $N\subset K$ and that $\dim N=k$.
It follows that $K=N$.

Finally suppose that $\mc B([v],[g])\in N$ for some $([v],[g])\in D_i\oplus D_{-i}$.
Abbreviating $E([v],[g])$ by $v$ we have $\sm{v(a)\\ v(b)}\in N$ and thus, by the first part of the proof, $([v],[g])\in\ov{T_\mn}$.
This implies $[v]=[g]=[0]$ and, since $E([0],[0])=0$ also $v(a)=v(b)=0$.
In particular, $\ker \mc B=\{0\}$ so that $\dim\ran\mc B=\dim(D_i\oplus D_{-i})$.
\end{proof}

As in Appendix \ref{ALR} we now define $\mc J:L^2(w)\exsum L^2(w): (u,f)\mapsto (f,-u)$ and recall that $\mc J(D_i\oplus D_{-i})=D_i\oplus D_{-i}$.
By $\bb J$ we shall denote the matrix $\sm{J&0\\0&-J}$.

\begin{thm}\label{T5.3}
Suppose $a$ and $b$ are regular endpoints.
Let $\tilde A\in\bb C^{n_+\times 2n}$ have the following properties: (1) $\tilde A$ has full rank $n_+$,
(2) $\bb J^{-1}\tilde A^*(\bb C^{n_+})\subset W$, and (3) $\tilde A\bb J^{-1}\tilde A^*=0$.
Then
$$T=\{([u],[f]): (u,f)\in\mc T_\mx,\ \tilde A \sm{u(a)\\u(b)}=0\}$$
is a self-adjoint restriction of $T_\mx$.

Conversely, every self-adjoint restriction $T$ of $T_\mx$ is determined in this way by a matrix $\tilde A$ with the given properties.
\end{thm}

\begin{proof}
With $\tilde A^*=\bb J\sm{g_1(a)& ...&g_{n_+}(a)\\ g_1(b)& ... &g_{n_+}(b)}$ this is an immediate consequence of Theorem~\ref{T:5.1}.
\end{proof}

\subsection{Green's function and the Fourier transform}
In the case of regular endpoints there are several simplifications in the process of constructing Green's function when compared to the one in Section \ref{sst.1}.
First of all we may choose $x_0$ to be one of the endpoints, say $a$.
Moreover, since all solutions of $Ju'+qu=\la wu$ are in $\mc L^2(w)$ the projections $P_\pm(\la)$ are both equal to the identity so that conditions \eqref{170811.1} and \eqref{170811.2} become vacuous.
Equation \eqref{170811.3}, collecting the $n_+=n_-$ boundary conditions, becomes
$$(g(b)^*J U(b,\la)-g(a)^*J\id)u_0 = -\int g(b)^*J U(b,\la)J^{-1} U(\cdot,\ol)^*wf.$$
Finally, taking, instead of all, only the linearly independent rows of $(\id-P)u_0=0$ into account we obtain an $n\times n$-matrix $F(\la)$.
As before $F(\la)$ has trivial kernel, \ie, it is invertible, if $\la$ is in the resolvent set.
We obtain $(E_\la f)(x)=\int U(x,\la)\tilde H(x,\cdot,\la)U(\cdot,\ol)^*wf$ where
$$\tilde H(x,y,\la)=M(\la)+\frac12(\id-P)J^{-1}P-\frac12 J^{-1}P\sgn(y-x)+S(x,\la)\chi_{\{x\}}(y)$$
with
$$M(\la)=-PF(\la)^{-1}\sm{g(b)^*\\ 0}J U(b,\la)J^{-1}P+\frac12 P J^{-1}P$$
and $S(x,\la)=\frac14U(x,\la)^{-1}(U^+(x,\la)-U^-(x,\la))J^{-1}$.

For the remainder of the section we will require the validity of Hypothesis \ref{H:6.1}, which reduces here to the condition $\Lambda\cap\bb R=\emptyset$.
As in Section \ref{sst.2} we denote the measure which occurs in the Herglotz-Nevanlinna representation of $M$ by $\nu$.
If $\la$ is an eigenvalue of multiplicity $k$, then Lemma \ref{L:6.6} implies that $\Delta_\nu(\la)$ is a matrix of rank $k$ and the columns of $U(x,\la)\Delta_\nu(\la)$ span the eigenspace associated with $\la$.

\begin{thm}
Suppose both endpoints of the interval $(a,b)$ are regular and $\Lambda\cap\bb R$ is empty.
If $T$ is a self-adjoint restriction of $T_\mx$, then its spectrum consists of isolated eigenvalues of finite multiplicity.
Their number may be finite or infinite and, in the latter case, they may accumulate only at $\infty$, $-\infty$ or both.
We denote by $Z$ a suitable subset of $\bb Z$ to label the eigenvalues of $T$, repeated according to their multiplicity, so that $\la_n \leq \la_{n+1}$.
The Fourier transform is then given by the sequence  $n\mapsto \hat f_n= \int U(\cdot,\ol_n)^*w f$ and, consequently, the inverse transform is a sum or series where each term is an eigenfunction of $T$.
Specifically,
$$f(x)=\sum_{n\in Z} U(x,\la_n)\Delta_\nu(\la_n)\hat f_n.$$
\end{thm}

\begin{proof}
We will prove that $R_\la$ is a compact operator as long as $\la\in\rho(T)$.
This means, by definition, that $\|R_\la f_n\|$ converges to $0$, when $f_n\in L^2(w)$ converges weakly to $0$.
Since, for fixed $x\in(a,b)$, the function $G(x,\cdot,\la)^*$ is in $L^2(w)$, it follows that $(E_\la f_n)(x)$ tends to $0$ as $n$ tends to infinity.
Next note that $u_n=R_\la f_n$ means that $(u_n,\la u_n+f_n)\in T$.
Since $Q$ and $W$ are of bounded variation, Lemma~\ref{L:6.2} and the fact that $R_\la$ is bounded give
$\vertiii{E_\la f_n}=\vertiii{E(u_n,\la u_n+f_n)}\leq C\|f_n\|$ for some appropriate constant $C$.
Since $f_n$ is weakly convergent we have that $\|f_n\|$ and hence $\vertiii{E_\la f_n}$ as well as $|(E_\la f_n)(x)|_1$ are bounded.
Thus, by the dominated convergence theorem,
$$\|R_\la f_n\|^2=\int (E_\la f_n)^* \RN (E_\la f_n)\tr w$$
tends to zero, since $\tr w$ is a finite measure and each component of $E_\la f_n$ and $\RN$, the Radon-Nikodym derivative of $w$ with respect to $\tr w$, is bounded.

A standard theorem in spectral theory for linear operators shows now that the operator part of $T$ has only isolated eigenvalues of finite multiplicity.
Lemma \ref{L:6.6} implies that the support of $\nu$ is equal to the collection of all eigenvalues $\{\la _n:n\in Z\}$ of $T$ and thus discrete.
In particular, the Fourier transform of $f\in L^2(w)$ is completely determined by the vectors $\hat f_n= \int U(\cdot,\ol_n)^*w f$, $n\in Z$.
Accordingly, the inverse transform is either a finite sum, a series on $\bb N$ or a series on $\bb Z$.
\end{proof}

\section{Special case: \texorpdfstring{$n=2$}{n=2} and real coefficients}\label{sn2}
We now specialize to the case $n=2$ and assume that the coefficients $J$, $q$, and $w$ are real.
This means that we must have $J=\beta\sm{0&-1\\ 1&0}$ with $\beta\in\bb R\setminus\{0\}$ (and we could restrict ourselves to the case $\beta=1$ by employing a coordinate transform).
It follows that $Ju'+qu=\la wu$ if and only if $J\ov{u}'+q\ov{u}=\ol w\ov{u}$ and therefore the deficiency indices $n_+$ and $n_-$ are the same and we have always self-adjoint restrictions of $T_\mx$.

\begin{lem}\label{L:7.2}
For each $\la\in\bb C\setminus \bb R$ and $c\in(a,b)$ there are non-trivial solutions of $Ju'+qu=\la wu$ in $L^2(w\chi_{(a,c)})$ and $L^2(w\chi_{(c,b)})$, respectively.
\end{lem}

\begin{proof}
We may assume that $\mc L_0$ is trivial since otherwise the claim is.
Let $U$ be a fundamental matrix for $Ju'+qu=0$ such that $U(x_0)=\id$ for some $x_0\in(a,b)$ and $[c,d]$ an interval in $(a,b)$ containing $x_0$ and such that the integrals $\int U^*w f$ reach all of $\bb C^2$ when $f$ varies among the elements of $L^2(w)$ supported in $[c,d]$.
This is possible according to Lemma \ref{L:5.1}.
We may also choose $c$ and $d$ to be points of continuity of $Q$ and $W$.
It follows now from the variation of constants formula (Lemma \ref{varconst2}) that there is an $f\in L^2(w)$ with $\supp f\subset[c,d]$ and a solution $u$ of $Ju'+qu=wf$ such that $u(c)=(1,0)^\top$ and $u(x)=0$ for $x>d$.
Similarly, there is a $g\in L^2(w)$ with $\supp g\subset[c,d]$ and a solution $v$ of $Jv'+qv=wg$ such that $v(c)=(0,1)^\top$ and $v(x)=0$ for $x>d$.
From Lemma \ref{T4.1.4a} we get $(v^*Ju)^+(a)=v(c)^*Ju(c)=\beta\neq0$.

Now consider the relations $\tilde T_\mn$ and $\tilde T_\mx$ determined by $J$, $q$, and $w\chi_{(c,b)}$.
Suppose that the deficiency indices of $\tilde T_\mn$ are $0$ implying that $\tilde T_\mx$ is self-adjoint.
Since $(u,f)$ and $(v,g)$ are in $\tilde T_\mx$ we get from Lagrange's identity \eqref{Lagrange} the absurdity that $0=\<v,f\>-\<g,u\>=v(a)^*Ju(a)\neq0$.
It follows that the deficiency spaces $\tilde D_\pm(\la)$ cannot be trivial.
Thus there is a solution of $Ju'+qu=\la wu$ on $(a,b)$ such that $\|u\chi_{(c,b)}\|<\infty$.
Replacing $c$ with any other number in $(a,b)$ in this last claim will not invalidate it.
A similar argument works, of course, with the roles of $a$ and $b$ reversed.
\end{proof}

\begin{lem}\label{L:7.1}
All solutions of $Ju'+qu=\la wu$ are in $L^2(w\chi_{(a,c)})$ for all $c\in(a,b)$ and all $\la\in\bb C\setminus\bb R$, if this holds true for one such pair $(c,\la)$.
The corresponding statement is also true for the endpoint $b$.
\end{lem}

\begin{proof}
Let $\tilde T_\mn$ be the relation determined by $J$, $q$, and $w\chi_{(a,c)}$.
Then our claim follows from the fact that the dimension of the deficiency spaces $\tilde D_\pm(\la)$ of $\tilde T_\mn$ does not vary with $\la$.
\end{proof}

For $c\in(a,b)$ and $\la\in\bb C\setminus\bb R$ we have now the following dichotomy: the space of solutions of $Ju'+qu=\la wu$ which lie in $L^2(w\chi_{(c,b)})$ is either one-dimensional or two-dimensional.
This characterization does not depend on the choice of $c$ or, in view of Lemmas \ref{L:7.1} and \ref{L:7.2}, on the choice of $\la\in\bb C\setminus\bb R$.
In the former situation we say that we have the \textit{limit-point case at $b$}\index{limit-point case} or, for short, that $b$ is limit-point, while in the latter situation we say that we have the \textit{limit-circle case at $b$}\index{limit-circle case} or that $b$ is limit-circle.
Of course, we have an analogous dichotomy for $a$ and, for any given problem, we have either the limit-point case or the limit-circle case at $a$.
The terminology goes back to Weyl's famous 1910 paper \cite{41.0343.01} on the subject.

\subsection{Definiteness condition violated}
Now assume $\dim\mc L_0=1$.
We have the limit-point case at $a$ or $b$ if and only if $n_+=n_-=0$ and the limit-circle case at both $a$ and $b$ if and only if $n_+=n_-=1$.

\begin{thm}
Suppose that $n=2$ and that $J$, $q$ and $w$ are real.
Furthermore, assume that $\dim\mc L_0=1$, $n_\pm=0$, and $\ker\Delta_w(x)\subset \ker\Delta_w(x)J^{-1}\Delta_q(x)$ for all $x\in (a,b)$,
Then $T_\mx=T_\mx^*=\{0\}\exsum L^2(w)$.
\end{thm}

\begin{proof}
We begin by showing that $PJ^{-1}P=0$.
If we denote the normalized vector which spans the range of $P$ by $m_0$ we get that $PJ^{-1}P=(m_0^*J^{-1}m_0)P$.
In our case $m_0=u(c)$ for an appropriately normalized non-trivial element $u$ of $\mc L_0$.
Since $\ov u$ is also in $\mc L_0$ we get that $m_0^*$ is a multiple of $m_0^\top$.
Hence $m_0^*J^{-1}m_0=0$.

Now assume that we have the limit-point case at $b$, the other case being treated similarly.
Suppose $(u,f)\in\mc T_\mx$ and $\la\in\bb C\setminus(\Lambda\cup\bb R)$.
Let $E$ be the evaluation operator from Definition \ref{D:5.2} with anchor point $x_0$.
Then $v=E_\la (f-\la u)$ is in $[u]$.
Since $P_+(\la)=\id-P$ and $U(\cdot,\ol)^*w=PU(\cdot,\ol)^*w$ equation \eqref{170811.1} shows that $Pu_0=-PJ^{-1}P\int_{(x_0,b)} U(\cdot,\ol)^*wf$ so that $Pu_0=0$.
However, equation \eqref{170811.4} states that $Pu_0=u_0$ implying $u_0=0$ in equation \eqref{170812.1}.

Let $n_0=J^{-1}m_0$ and note that this spans $N_0$.
Using Lemma \ref{L:5.1} we see that both $\int_{(x_0,x)}U(\cdot,\ol)^*wf$ and $-\int_{(x,x_0)}U(\cdot,\ol)^*wf$ are scalar multiples of $m_0$.
Similarly $U(x,\ol)^*\Delta_w(x)f(x)$ is a scalar multiple of $m_0$.
Hence we get,
$$v(x)=(\alpha(x)+\beta(x)(\id\pm J^{-1}(\la \Delta_w(x)-\Delta_q(x)))) U(x,\la)n_0$$
for appropriate scalar-valued functions $\alpha$ and $\beta$.
It follows that $wv$ is a discrete measure.
In fact, $(wv)(\{x\})=-\beta(x)\Delta_w(x)J^{-1}\Delta_q(x) U(x,\la)n_0$.
Since $U(x,\la) n_0$ is in $\ker\Delta_w(x)$, it is, by assumption, also in $\ker\Delta_w(x)J^{-1}\Delta_q(x)$.
It follows that $wv$ is the zero measure and hence that $v\in[0]$, \ie, $[u]=[0]$.
We have now shown that $\dom T_\mx$ is trivial.
That $\ran T_\mx=L^2(w)$ follows from the self-adjointness of $T_\mx$.
\end{proof}

We note that the condition $\ker\Delta_w(x)\subset \ker\Delta_w(x)J^{-1}\Delta_q(x)$ is satisfied if, for instance, at most one of $Q$ and $W$ jumps at any given point $x$.

\subsection{Definiteness condition holds}
This case closely resembles that of the classical Sturm-Liouville equation and it is possible to follow the well-known blueprint for that situation.
Specifically we emulate Chapters 3 -- 6 of Eckhardt et al. \cite{MR3046408}.

First note that, if $(u,f)$ and $(v,g)$ are in $\mc T_\mx$ such that $f$ and $g$ are equivalent in $L^2(w)$ then $u-v$ satisfies the equation $J(u-v)'+q(u-v)=0$.
Thus $u-v\in\mc L_0=\{0\}$, if $u$ and $v$ are also equivalent.
This means that, given an element $([u],[f])\in T_\mx$, we may find a unique $u\in\dom\mc T_\mx$ such that $Ju'+qu=wf$.

Another  major consequence of the definiteness condition is that one may modify elements of $T_\mx$ on part of the interval $(a,b)$ in certain ways without leaving the set.
More specifically, we have the following lemma.

\begin{lem}\label{L:8.4}
Suppose $Ju'+qu=wf$ and $u\in L^2(w\chi_{(a,c)})$ for some (and hence any) $c\in(a,b)$.
Then there is a $(v,g)\in T_\mx$ such that $v=u$ near $a$ and $v=0$ near $b$.
A similar statement holds with the roles of $a$ and $b$ reversed.
\end{lem}

\begin{proof}
In view of the definiteness condition there is an interval $[c,d]\subset (a,b)$ such that $wU(\cdot,0)(\alpha,\beta)^\top\neq0$ on $(c,d)$ for all nontrivial $(\alpha,\beta)^\top\in\bb C^2$.
In fact we may choose both $c$ and $d$ to be points of continuity for $Q$ and $W$.
Let $U$ be a fundamental matrix for $Ju'+qu=0$ so that $U(c)$ is the identity matrix.
According to Lemma \ref{L:5.1} we can find a $g:(c,d)\to\bb C$ such that $g\in L^2(w\chi_{(c,d)})$ and $u(c)+J^{-1}\int_{(c,d)} U^*wg=0$.
We then define $v$ on $[c,d]$ by $v^-(x)=U^-(x)(u(c)+J^{-1}\int_{(c,x)} U^*wg$ as well as $v=u$ on $(a,c]$ and $v=0$ on $[d,b)$.
We also extend $g$ to $(a,c)$ and $(d,b)$ by $f$ and $0$, respectively.
Then $(v,g)\in T_\mx$.
\end{proof}

The first consequence of this is that we may characterize $\ov{T_\mn}$.

\begin{thm}\label{T:8.5}
$$\ov{T_\mn}=\{(u,f)\in T_\mx: \forall (v,g)\in T_\mx: (v^*Ju)^-(b)=(v^*Ju)^+(a)=0\}.$$
\end{thm}

\begin{proof}
If $(u,f)\in\ov{T_\mn}=T_\mx^*$, it is clear from Lagrange's identity \eqref{Lagrange} that $(v^*Ju)^-(b)-(v^*Ju)^+(a)=0$ whenever $(v,g)\in T_\mx$.
Now find, using Lemma \ref{L:8.4}, a pair $(v_0,g_0)\in T_\mx$ such that $v_0=0$ near $a$ and $v_0=v$ near $b$.
Then $(v^*Ju)^-(b)=(v_0^*Ju)^-(b)=(v_0^*Ju)^-(b)-(v_0^*Ju)^+(a)=0$.
For the converse use Lagrange's identity again.
\end{proof}

Occasionally we will use below the abbreviation $V$ for $D_i\oplus D_{-i}$.

\begin{lem}\label{L:8.6}
If we have the limit-point case at $b$, then $(r^*Ju)^-(b)=0$ for all $(u,f),(r,h)\in T_\mx$.
The analogous statement is true when $a$ is limit-point.
\end{lem}

\begin{proof}
Suppose first that $a$ is a regular point, $(u,f)\in T_\mx$, and $c\in(a,b)$.
According to Lemma \ref{L:7.2} there is one solution of $Ju'+qu=iwu$ which is in $L^2(w)$ but there cannot be two since $b$ is limit-point.
Hence we have $n_\pm=1$ and $\dim V=2$.
Now let $v_1$ and $v_2$ be solutions of $Ju'+qu=0$ satisfying initial conditions $v_1(a)=(1,0)^\top$ and $v_2(a)=(0,1)^\top$.
By Lemma \ref{L:8.4}, there are elements $(\tilde v_k,f_k)\in T_\mx$ such that $\tilde v_k$ agrees with $v_k$ near $a$ but is zero near $b$.
We then have $(\tilde v_2^*J \tilde v_1)(a)\neq0$ proving, in view of Theorem \ref{T:8.5}, that neither $(\tilde v_1,f_1)$ nor $(\tilde v_2,f_2)$ is in $\ov{T_\mn}$.
Thus $T_\mx=\ov{T_\mn}\dot+\span\{(\tilde v_1,f_1),(\tilde v_2,f_2)\}$.
It follows that any element of $T_\mx$ agrees with some element of $\ov{T_\mn}$ near $b$.
Using Lemma \ref{L:8.4} again, we may find, for any $(r,h)\in T_\mx$, an element $(\tilde r,\tilde h)\in T_\mx$ such that $\tilde r$ equals $r$ near $b$ and $0$ near $a$.
Lagrange's identity \eqref{Lagrange} gives now $(r^*Ju)^-(b)=(\tilde r^*J\tilde u)^-(b)-(\tilde r^*J\tilde u)^+(a)=\<\tilde r,\tilde f\>-\<\tilde h,\tilde u\>=0$.

For the general case consider the problem posed on the interval $(a',b)$ for some $a'\in(a,b)$.
\end{proof}

These considerations give now that we have the following trichotomy just as in the classical case.
\begin{thm}\label{T:8.7}
The following statements hold true.
  \begin{enumerate}
    \item $n_+=n_-=0$ if and only if we have the limit-point case at both $a$ and $b$.
    In this situation $T_\mx$ is self-adjoint.
    \item $n_+=n_-=1$ if and only if we have the limit-point case at one of $a$ and $b$ and the limit-circle case at the other.
    Self-adjoint restrictions are given by posing a boundary condition at the limit-circle end.
    \item $n_+=n_-=2$ if and only if we have the limit-circle case at both $a$ and $b$.
  \end{enumerate}
\end{thm}

\begin{proof}
Suppose we have the limit-point case at both $a$ and $b$.
For any $(u,f),(v,g)\in T_\mx$ Lemma \ref{L:8.6} and Lagrange's identity \eqref{Lagrange} give $\<g,u\>=\<v,f\>$.
Hence $T_\mx$ is self-adjoint and the deficiency indices of $T_\mn$ are zero.
Conversely, if at least one of $a$ and $b$ is limit-circle Lemma \ref{L:7.2} gives that $n_\pm>0$.
This proves (1).

Now suppose that we have the limit-circle case at both $a$ and $b$.
Then every solution of $Ju'+qu=\pm iw u$  is in $L^2(w)$, \ie, $n_\pm=2$.
Conversely, if at least one $a$ and $b$ is limit-point we must have $n_\pm<2$.
This proves (3) and thus also the first statement of (2).

For the second statement of (2) assume that the boundary condition is given by $(v,g)\in V$ (cf. Theorem \ref{T:5.1}) and assume that $a$ is limit-circle and $b$ is limit-point (the other case being treated similarly).
Lemma \ref{L:8.6} gives $(g^*Ju)^-(b)=0$ for any $(u,f)\in T_\mx$ so that the boundary condition reads $(g^*Ju)^+(a)=0$.
\end{proof}

Lemma \ref{L:8.6} says that boundary conditions have no effect at endpoints which are limit-point, they only affect limit-circle endpoints.
We saw in Section \ref{sre} that we may express boundary conditions through boundary values at regular endpoints.
Something similar is true under the present circumstances for limit-circle endpoints as we will see next.
For this we assume that we have the limit-circle case at $a$ but the same arguments work, of course, when $b$ is limit-circle.
The key to this is the identity
$$\sum_{j=1}^3 (-1)^j (A^*JB_j)(B_k^*JB_\ell)=0$$
which holds for arbitrary vectors $A,B_1,B_2,B_3\in\bb C^2$ when $j,k,\ell$ are pairwise different and $k<\ell$.
It is straightforward to check the validity of this identity but we note that it is one of the Gra\ss mann-Pl\"ucker relations\footnote{It is customary to express these relations in terms of determinants and we are making use of the fact that $A^*JB=-\beta\det(\ov{A},B)$.}.

We will choose $A=g(x)$, $B_1=v_1(x)$, $B_2=v_2(x)$, and $B_3=u(x)$, where $g,u\in\dom T_\mx$ and $v_1,v_2$ are real-valued, linearly independent solutions of $Ju'+qu=0$.
Then we obtain
\begin{multline}\label{180114.1}
(g(x)^*Ju(x))(v_1(x)^*Jv_2(x))\\ =(g(x)^*Jv_2(x))(v_1(x)^*Ju(x))-(g(x)^*Jv_1(x))(v_2(x)^*Ju(x)).
\end{multline}
Each of the terms appearing here has a limit as $x$ tends to $a$ thanks to Lagrange's identity \eqref{Lagrange} and Lemma \ref{L:8.4}.
Moreover, $v_1^+(x)^*Jv_2^+(x)$ is different from zero and constant in view of Lemma~\ref{T4.1.4a}.
We may choose this constant to be $-1/\beta$.
Upon defining
$$\vec u(a)=((v_1^*Ju)^+(a),(v_2^*Ju)^+(a))^\top\text{ and }\vec g(a)=((v_1^*Jg)^+(a),(v_2^*Jg)^+(a))^\top$$
equation \eqref{180114.1} is equivalent to
\begin{equation}\label{180114.2}
(g^*Ju)^+(a)=\vec g(a)^*J\vec u(a).
\end{equation}
Of course similar considerations work at $b$ if it is limit-circle.
We emphasize that we may choose different solutions $v_1$ and $v_2$ near either of the endpoints should they both be limit-circle.
If $a$ is regular we may choose $v_1$ and $v_2$ so that $v_1(a)=(0,1/\beta)^\top$ and $v_2(a)=(-1/\beta,0)^\top$ to get $\vec u_1(a)=u_1(a)$ and $\vec u_2(a)=u_2(a)$.

Now we return to case (2) of Theorem \ref{T:8.7}.
Assume $a$ is limit-circle and $b$ is limit-point (the other case being similar).
According to Theorem \ref{T:5.1} and Lemma \ref{L:8.6} any self-adjoint restriction of $T_\mx$ is given by
$$T=\{(u,f)\in T_\mx: (g^*Ju)^+(a)=0\}=\{(u,f)\in T_\mx: \vec g(a)^*J\vec u(a)=0\}$$
where $(v,g)\in V$ satisfies $\<(v,g),(g,-v)\>=2i\Im(\<v,g\>)=0$.
Here we may choose the components of $\vec g(a)^*J$ as $(\cos(\alpha),-\sin(\alpha))$ for some $\alpha\in[0,\pi)$ and obtain
\begin{equation}\label{180216.1}
T=\{(u,f)\in T_\mx: \cos(\alpha) \vec u_1(a)-\sin(\alpha)\vec u_2(a)=0\}.
\end{equation}
In other words, the self-adjoint restrictions of $T_\mx$ are in one-to-one correspondence to points in $[0,\pi)$.

Next we turn to case (3) of Theorem \ref{T:8.7} where both $a$ and $b$ are limit-circle endpoints.
Now Theorem \ref{T:5.1} states that the self-adjoint restriction of $T_\mx$ are determined by two elements $(v_1,g_1),(v_2,g_2)\in V$ satisfying
\begin{equation}\label{180114.3}
\<(v_k,g_k),(g_j,-v_j)\>=0\text{ for $j,k=1,2$}.
\end{equation}
In fact, the boundary condition is $A(u,f)=0$ where $A_j=\<(v_j,g_j),\cdot\>$, $j=1,2$.
Using Lagrange's identity \eqref{Lagrange} we get $A(u,f)=\tilde A \sm{\vec u(a)\\ \vec u(b)}$, where
$$\tilde A=\begin{pmatrix}\vec g_1(a)^*&\vec g_1(b)^*\\ \vec g_2(a)^*&\vec g_2(b)^*\end{pmatrix}\bb J^*$$
with, as in Section \ref{sre}, $\bb J=\sm{J&0\\0&-J}$.
Condition \eqref{180114.3} becomes $\tilde A^*\bb J^{-1}\tilde A=0$.
Note that $\tilde A$ is a matrix in $\bb C^{2\times 4}$ which has (full) rank 2.
Hence we can now, similar to Theorem \ref{T5.3}, state that the self-adjoint restriction $T$ of $T_\mx$ are precisely given by
$$T=\{(u,f)\in T_\mx: \tilde A \sm{\vec u(a)\\ \vec u(b)}=0\}$$
where $\tilde A\in\bb C^{2\times 4}$ has rank $2$ and satisfies $\tilde A\bb J^{-1}\tilde A^*=0$.

To investigate the various possibilities we now write $\tilde A$ as two blocks of $2\times 2$ matrices, \ie, $\tilde A=(A_1,A_2)$.
Thus the condition $\tilde A\bb J^{-1}\tilde A^*=0$ becomes $A_1J^{-1}A_1^*=A_2J^{-1}A_2^*$.
If $A_1$ is invertible we may as well assume that it is the identity.
It follows that $A_2$ must be invertible and hence that the boundary conditions must be coupled (neither can involve $\vec u(a)$ or $\vec u(b)$ alone).
It follows similarly that $A_1$ is invertible if $A_2$ is, leading again to coupled boundary conditions.

It remains to consider the case where both $A_1$ and $A_2$ have rank $1$.
In this case we may assume that one of the rows of $A_1$ is equal to zero.
The corresponding row of $A_2$ must then be different from zero.
In fact, we may assume without loss of generality that $\tilde A=\sm{A_{1,1}&A_{1,2}&0&0\\ 0&0&A_{2,1}&A_{2,2}}$ so that we have separated boundary conditions. Thus, mixed boundary conditions cannot occur.

We finish the chapter by investigating the relationship between the matrix $M$ introduced in Section \ref{sst.2} and the Titchmarsh-Weyl function for the case when $a$ is regular and $b$ is limit-point.
We also set $\beta=1$ and assume for the remainder of the chapter the validity of Hypothesis \ref{H:6.1}.
Let $x_0=a$ and let $U(\cdot,\la)$ be the fundamental matrix for $Ju'+qu=\la wu$ satisfying $U(a,\la)=\id$.
We set $\theta(\cdot,\la)=U(\cdot,\la)(\cos\alpha,-\sin\alpha)^\top$ and $\varphi(\cdot,\la)=U(\cdot,\la)(\sin\alpha,\cos\alpha)^\top$ so that $\varphi(\cdot,\la)$ is a solution satisfying the boundary conditions in \eqref{180216.1} at $a$.
The Titchmarsh-Weyl function $m$ is now defined by the requirement that $\theta(\cdot,\la)+m(\la)\varphi(\cdot,\la)$ is in $L^2(w)$.
The function $\theta(\cdot,\la)+m(\la)\varphi(\cdot,\la)$ itself is called the Titchmarsh-Weyl solution of $Ju'+qu=\la wu$.

If $\la$ is in $\rho(T)\setminus\Lambda$ and $f$ supported in $[c,d]\subset(a,b)$ we have that
$$(E_\la f)(x)=\begin{cases}U(x,\la)(M(\la)-\frac12 J^{-1}) \int_{[c,d]} U(\cdot,\ol)^*wf&\text{if $x<c$}\\
                            U(x,\la)(M(\la)+\frac12 J^{-1}) \int_{[c,d]} U(\cdot,\ol)^*wf&\text{if $x>d$}\end{cases}$$
is in $L^2(w)$.
Since $\int U(\cdot,\ol)^*wf$ spans $\bb C^2$ as $f$ varies, when $[c,d]$ is sufficiently large, it follows that any linear combination of the columns of $U(\cdot,\la)(M(\la)+\frac12 J^{-1})$ is in $L^2(w)$.
Also, it is a linear combination of $\theta(\cdot,\la)$ and $\varphi(\cdot,\la)$, \ie,
$$U(x,\la)(M(\la)+\frac12 J^{-1})\begin{pmatrix}\sin\alpha\\ \cos\alpha\end{pmatrix}
 =(\theta(x,\la),\phi(x,\la))\begin{pmatrix}c_1\\ c_2\end{pmatrix}$$
for suitable $c_1$ and $c_2$.
Solving for $(c_1,c_2)^\top$ gives
$$\begin{pmatrix}c_1\\ c_2\end{pmatrix}
 =\begin{pmatrix}\cos\alpha&-\sin\alpha\\ \sin\alpha&\cos\alpha\end{pmatrix}
 (M(\la)+\frac12 J^{-1})\begin{pmatrix}\sin\alpha\\ \cos\alpha\end{pmatrix}.$$

Since $E_\la f$ satisfies the boundary condition and $U(a,\la)=\id$ we have
$$(\cos\alpha,-\sin\alpha)(M(\la)-\frac12 J^{-1})\begin{pmatrix}\sin\alpha\\ \cos\alpha\end{pmatrix}=0$$
so that $c_1=1$, $c_2=m$, and
$$m(\la)=(\sin\alpha,\cos\alpha)M(\la)\begin{pmatrix}\sin\alpha\\ \cos\alpha\end{pmatrix}.$$
Note that $m$ is a Herglotz-Nevanlinna function.

\section{Examples}\label{sbsp}
We will now discuss two examples examining separately the effects of having $\Lambda$ non-empty and a missing definiteness condition.

\begin{comment}
$(a,b)=(0,1)$, $n=2$ and $w=\sm{1&0\\ 0&0}$
\begin{enumerate}
  \item $J=\sm{0&1\\ -1&0}$ and $q=\sm{0&0\\ 0&-1}$. In this case $T_\mx$ is an operator and the definiteness condition holds.
  \item $J=\sm{0&1\\ -1&0}$ and $q=\sm{0&0\\ 0&0}$. In this case $T_\mx$ is not an operator and the definiteness condition does not hold.
  \item $J=i \id_2$ and $q=\sm{0&0\\ 0&0}$. In this case $T_\mx$ is an operator but the definiteness condition does not hold.
  \item $J=i \id_2$ and $q=\sm{0&1\\ 1&0}$. In this case $T_\mx$ is not an operator but the definiteness condition holds.
\end{enumerate}
\end{comment}

\subsection{Example I}
Suppose $\iOmega=\bb R$, $n=1$, $J=i$, $q=0$, and $w=\delta_0$.
$\mc L^2(w)$ is the space of all complex-valued functions on $\bb R$.
To form $L^2(w)$ one identifies any two such functions if their values at $0$ agree.
As a consequence $L^2(w)$ is one-dimensional.
We have $\Lambda=\{\pm2i\}$.
Solutions of $Ju'=wf$ have to be constant on both $(-\infty,0)$ and $(0,\infty)$.
Denoting these values by $u_\ell$ and $u_r$, respectively, we have that $i(u_r-u_\ell)=f(0)$ and $u_r+u_\ell=2u(0)$.
Since this system has solutions for arbitrary values of $u(0)$ and $f(0)$ we have $T_\mx=L^2(w)\exsum L^2(w)$.
However, to have a solution of $Ju'=wf$ compactly supported requires $u_\ell=u_r=u(0)=f(0)=0$ so that $T_\mn=\{(0,0)\}$.
Note that $T_\mn$ is not densely defined and that $T_\mx$ is not an operator.
As relations, however, $T_\mn^*=T_\mx$, in agreement with the Theorem \ref{T4.1.6}.

Choosing $x_0<0$ a fundamental matrix is given by
$$U(x,\la)=
\begin{cases}1&\text{if $x<0$}\\ 2i/(2i-\la)&\text{if $x=0$}\\ (2i+\la)/(2i-\la)&\text{if $x>0$}\end{cases}$$
when $\la\neq\pm2i$.
It follows that $\mc L_0$ is trivial. Consequently, given an element $([u],[f])\in T_\mx$, we may find a unique $u\in\dom\mc T_\mx$ such that $Ju'+qu=wf$.

Note that $D_\la$ has dimension $1$ (even when $\la=\pm2i$) so that the deficiency indices are both equal to $1$.
In particular, the spaces $D_{\pm i}$ are spanned respectively by the vectors $(u_\pm,\pm iu_\pm)$, where
$u_+(x)=\big\{\begin{smallmatrix}1&\text{if $x<0$} \\ 3&\text{if $x>0$}\end{smallmatrix}$ and
$u_-(x)=\big\{\begin{smallmatrix}3&\text{if $x<0$} \\ 1&\text{if $x>0$}\end{smallmatrix}$.
Any boundary condition giving rise to a self-adjoint relation is given by a function $g(x)=i\big\{\begin{smallmatrix}\alpha-3\beta&\text{if $x<0$} \\ 3\alpha-\beta&\text{if $x>0$}\end{smallmatrix}$ where $|\alpha|=|\beta|\neq0$.
Then
$$(g^*Ju)^-(\infty)-(g^*Ju)^+(-\infty)=(3\ov\alpha-\ov\beta)(u_r-\frac{\ov\alpha-3\ov\beta}{3\ov\alpha-\ov\beta}u_\ell).$$
Note that the number $\gamma=(\ov\alpha-3\ov\beta)/(3\ov\alpha-\ov\beta)$ has absolute value $1$.
Thus, for a given $\gamma$ on the unit circle, we define now
$$T=\{(u,f)\in T_\mx: u_r=\gamma u_\ell\}$$
to obtain a self-adjoint restriction of $T_\mx$.
In fact, all self-adjoint restriction of $T_\mx$ are obtained this way.
If $\gamma=-1$, \ie, $\alpha=\beta$, then $T=\{0\}\exsum L^2(w)$ and the operator part $T_0$ of $T$ is given by $\{(0,0)\}$ since $\mc H_\infty=L^2(w)$.
Otherwise, if $\gamma\neq-1$, we have $\mc H_\infty=\{0\}$ and $T=T_0$ is the operator of multiplication by the real number $\la_0=2i(\gamma-1)/(\gamma+1)$.

To determine Green's function recall that $x_0<0$ and note that conditions \eqref{170811.1}, \eqref{170811.2}, and \eqref{170811.4} are irrelevant.
Only equation \eqref{170811.3} serves to determine $u_0$ and hence $H_\pm$.
In fact we obtain $A_+(\la)=(3\ov\alpha-\ov\beta)(2i+\la)/(2i-\la)$ and $A_-(\la)=(3\ov\alpha-\ov\beta)(2i+\la_0)/(2i-\la_0)$ so that
$$F(\la)=-4i(3\ov\alpha-\ov\beta)\frac{\la_0-\la}{(2i-\la_0)(2i-\la)}$$
and
$$H_\pm(\la)=\frac{(\la\pm2i)(\la_0\mp2i)}{4(\la_0-\la)}.$$
However, since $w$ has no support on $(-\infty, x_0]$ (implying that $B_-=\{0\}$) we have no need for $H_-$.
Now, if $\la$ is in the resolvent set of $T$, \ie, if $\la\neq\la_0$, and $f\in L^2(w)$, we obtain
$$(E_\la f)(x)=\frac{f(0)}{\la_0-\la}
\begin{cases} 1+i\la_0/2&\text{for $x<0$}\\ 1&\text{for $x=0$}\\ 1-i\la_0/2&\text{for $x>0$}.\end{cases}$$

The $M$-function becomes
$$M(\la)=\frac{4+\la\la_0}{4(\la_0-\la)}=\frac{3\la_0}{4(\la_0^2+1)}+\int \big(\frac1{t-\la}-\frac{t}{t^2+1}\big)\nu(t)$$
where $\nu=\frac14(\la_0^2+4) \delta_{\la_0}$.

The Fourier transform is given by $(\mc Ff)(t)=2i f(0)/(2i+t)$.
Of course, only its value at $t=\la_0$ is relevant.
Its adjoint (or inverse) is given by $(\mc G\hat f)(x)=\frac14 (\la_0^2+4)U(x,\la_0)\hat f(\la_0)$.

\subsection{Example II}
Let $n=2$, $a=0$, $b\in(0,\infty)$, $J=\sm{0&-1\\ 1&0}$, $q=0$, and $w=\sm{1&0\\ 0&0}$.
Then $\mc L^2(w)$ consists of pairs of complex-valued functions on $(0,b)$ whose first components are square integrable.
Any two such functions are equivalent if their first components are equal almost everywhere.
The corresponding classes form $L^2(w)$ which is infinite-dimensional.
In this case $\Lambda$ is empty.

The relation $\mc T_\mx$ is made up of elements $(u,f)$ where
$$u=\begin{pmatrix}c_1\\ c_2-\int_0^x f_1\end{pmatrix}
 \text{ and }
 f=\begin{pmatrix}f_1 \\f_2 \end{pmatrix}$$
with $c_1,c_2\in\bb C$, $f_1\in \mc L^2(0,b)$ and $f_2$ is completely arbitrary.
$(u,f)$ is in $\mc T_\mn$ if $c_1=c_2=0$, $\supp f_1$ compact and $\int_0^b f_1=0$.
Thus $T_\mx$ is equivalent to $\bb C\exsum L^2(w)$ while $\ov{T_\mn}=\{0\}\exsum\{f\in L^2(w): \int_0^b f_1=0\}$.

Fix $x_0\in[0,b]$.
Then the fundamental matrix, for which $U(x_0,\la)=\id$, is given by
$$U(x,\la)=\begin{pmatrix}1&0\\ \la (x_0-x)&1\end{pmatrix}.$$
Either column of $U(\cdot,\la)$ is a non-trivial solution of $Ju'+qu=\la w u$.
While the first column has positive norm the second has norm zero.
Thus $\mc L_0$ is one-dimensional so that the definiteness condition is violated.
The projection $P$ is given by $\sm{1&0\\ 0&0}$.
It also follows that the deficiency indices are both equal to $1$.

Any boundary condition is determined by an element $([v],[g])\in D_i\oplus D_{-i}$, \ie,
$$(v,g)=\begin{pmatrix}\alpha+\beta&i(\alpha-\beta)\\ -i(\alpha-\beta)x&(\alpha+\beta)x\end{pmatrix}$$
for which $|\alpha|=|\beta|\neq0$.
It reads
$$0=(g^*Ju)(b)-(g^*Ju)(0)=b(\ov{\alpha}+\ov{\beta}) u_1(b)+i(\ov{\alpha}-\ov{\beta})(u_2(b)-u_2(0)).$$
If $\alpha=\beta$ we have a separated boundary condition and obtain $T=\{0\}\exsum L^2(w)$.
Thus $\mc H_\infty=L^2(w)$ and $T_0=\{(0,0)\}$.
Otherwise, if $\alpha\neq\beta$ we have a coupled boundary condition and the self-adjoint restrictions of $T_\mx$ are determined by
$$\mc T=\{(u,f)\in \mc T_\mx:\gamma u_1(b)+u_2(b)-u_2(0)=0\}$$
where $\gamma=-ib(\ov{\alpha}+\ov{\beta})/(\ov{\alpha}-\ov{\beta})$ is an arbitrary real number.
In this case we have $\mc H_\infty=\{f\in L^2(w):(0,f)\in T\}=\{f\in L^2(w):\int_0^b f_1=0\}$ and $\mc H_0=L^2(w)\ominus\mc H_\infty$
consists of the classes represented by functions with a constant first component.
We have $T_0=T\cap(\mc H_0\exsum \mc H_0)=D_{\la_0}$ where $\la_0=\gamma/b$.
We emphasize that the boundary conditions are well defined for $([u],[f])\in T_\mx$ even though the second components of the boundary values are not.

Since we have regular endpoints let us also illustrate Theorem \ref{T5.3}.
The space $N$ introduced in Section \ref{sre} is spanned by $(0,1,0,1)^\top$.
Since $E([v],[g])=(\alpha+\beta,i(\alpha-\beta)(x_0-x))^\top$ we find that $W$ is spanned by the vectors $(1,\pm ix_0,1,\mp i(b-x_0))^\top$.
If $\tilde A=(a_1,a_2,a_3,a_4)$ condition (2) in Theorem \ref{T5.3} requires $(\ov a_2,-\ov a_1,-\ov a_4,\ov a_3)^\top$ to be in $W$ and hence we get $x_0(a_1+a_3)=b a_1$ and $a_4=-a_2$.
Using these requirements, the condition $\tilde A\bb J^{-1}\tilde A^*=0$, \ie, condition (3), gives $\Im((a_1+a_3)\ov{a_2})=0$ so that without loss of generality, we may assume that the $a_j$ are real.
This, and the fact that the first component of $u$ is constant gives, as before, the boundary condition $(a_1+a_3)u_1+a_2(u_2(b)-u_2(0))$.

We now determine $H_\pm$ for the case where $\alpha\neq\beta$.
First note that conditions \eqref{170811.1} and \eqref{170811.2} are again irrelevant.
Since $A_+(\la)=i(\ov\alpha-\ov\beta)(b(\la_0-\la)+x_0\la,1)$, $A_-(\la)=i(\ov\alpha-\ov\beta)(x_0\la,1)$, and $PJ^{-1}P=0$, we obtain
$$H_-(\la)=\frac{1/b}{\la_0-\la}\begin{pmatrix}1&-x_0\la\\ 0&0\end{pmatrix}\text{ and }
  H_+(\la)=\frac{1/b}{\la_0-\la}\begin{pmatrix}1&-b(\la_0-\la)-x_0\la\\ 0&0\end{pmatrix}.$$

Thus, if $\la$ is in the resolvent set of $T$, \ie, if $\la\neq\la_0$, and $f\in L^2(w)$,
$$(E_\la f)(x)=\frac{\int_0^b f_1}{b(\la_0-\la)}\begin{pmatrix}1\\ \la(x_0-x)\end{pmatrix}-\int_{x_0}^x f_1\begin{pmatrix}0\\ 1\end{pmatrix}.$$

The $M$-function becomes
$$M(\la)=\frac1{b(\la_0-\la)}P=\frac{\la_0}{b(\la_0^2+1)}P+\int \big(\frac1{t-\la}-\frac{t}{t^2+1}\big)\nu(t)$$
where $\nu=\frac1b\delta_{\la_0}P$.
We emphasize that for a fixed non-real $\la$ the points $1/(b(\la_0-\la))$ lie on a circle centered at $z_0=i/(2b\Im(\la))$ of radius $|z_0|$ when $\la_0$ is on the real line.
As $b$ tends to infinity this circle shrinks to the point $0$.

The Fourier transform is given by $(\mc Ff)(\la)=\int U(\cdot,\ol)^*wf=(1,0)^\top\int f_1$.
Note that $\ker\mc F=\mc H_\infty$.

Since the spectral projections are
$$S(\omega)=\begin{cases}\bb P&\text{if $\la_0\in\omega$}\\ 0&\text{if $\la_0\not\in\omega$,}\end{cases}$$
when $\bb P$ is the orthogonal projection from $L^2(w)$ to $\mc H_0$,
we obtain
$$\<f,S(-\infty,t)g\>_w=\int\chi_{(-\infty,t)}\hat f^*\nu\hat g$$
for all $f,g\in L^2(w)$.
It follows that $\<f,\bb P g\>=\<\mc Ff,\mc Fg\>_\nu$, \ie, $\mc F:\mc H_0\to L^2(\nu)$ is unitary.

Define $(\mc G\hat f)(x)=\int U(x,\cdot)\nu \hat f= U(x,\la_0)P\hat f(\la_0)/b$ when $\hat f\in L^2(\nu)$ and note that this is in $\dom\mc T_\mx$.
In fact, it follows that $\mc F^* \hat f=[\mc G\hat f]$.
In particular, if $\hat f=\mc F f$, then $(\mc G\hat f)(x)=\frac1b \int f_1 (1,0)^\top\in\bb P [f]$.
Moreover, $(\mc F f)(t)=t (\mc Fu)(t)$, if $(u,f)\in T$.

\begin{comment}
Finally, we discuss briefly the case where $b=\infty$.
In this case the elements $(u,f)$ of $\mc T_\mx$ are
$$u=\begin{pmatrix}0\\ c_2+\int_0^x f_1\end{pmatrix}
 \text{ and }
 f=\begin{pmatrix}f_1 \\f_2 \end{pmatrix}$$
with $c_2\in\bb C$ and $f_1\in \mc L^2(0,b)$.
The relation $T_\mx$ is now self-adjoint since the solutions of $Ju'+qu=\la wu$ have norm zero or infinity, \ie, $n_\pm=0$.

We now have $P_+(\la)=\id-P=\sm{0&0\\ 0&1}$ and we assume for simplicity $c=0$.
Conditions \eqref{170811.2} and \eqref{170811.3} are now irrelevant while conditions \eqref{170811.1} and \eqref{170811.4} give
$$u_0=\sm{-1&0\\ 0&0}J^{-1}\int_0^\infty U(\cdot,\ol)^*wf.$$
Thus $H_+(\la)=\id$
\end{comment}

\appendix
\section{Distributions and measures}\label{ADM}
We collect here the most basic facts about distributions on real intervals and their relationship with measures.
In the case of distributions these may be found, for instance, in the books by Gelfand and Shilov \cite{MR0435831} or H\"ormander \cite{MR1996773}.
Locally distributions of order $0$ may be identified with measures whose theory as expounded by, for instance, Rudin \cite{MR924157} and Folland \cite{MR1681462} we assume known.

\subsection{Distributions}\label{ADM.1}
The space of complex-valued functions defined on $\iOmega$ which have derivatives of all orders and are supported on compact subsets of $\iOmega$, is denoted by $\mc D(\iOmega)$.
These functions are called {\em test functions}\index{test functions}.
A linear functional $q$ on $\mc D(\iOmega)$ is called a {\em distribution}\index{distribution} on $\iOmega$ if for every compact set $K\subset\iOmega$ there are constants $C>0$ and $k\in\bb N_0$ such that
\begin{equation}\label{eq:1.1.1}
|q(\phi)|\leq C \sum_{j=0}^k \sup\{|\phi^{(j)}(x)|:x\in K\}
\end{equation}
whenever the test function $\phi$ has its support in $K$.
The set of all distributions on $\iOmega$ is denoted by $\mc D'(\iOmega)$.
If the integer $k$ in \eqref{eq:1.1.1} can be chosen uniformly for every compact $K\subset\iOmega$, then $q$ is said to have finite order.
The smallest such integer is called the {\em order}\index{order of a distribution} of $q$.
The set of distributions of order at most $k$ is denoted by $\mc D^{\prime k}(\iOmega)$.

The most basic example of a distribution is given by the map $d_f:\phi\mapsto \int\phi f$ where $f$ is a locally integrable complex-valued function on $\iOmega$ and integration is with respect to Lebesgue measure.\footnote{In view of this, we will often identify locally integrable functions with the distributions they generate and, more generally,  write $\int\phi q$ or $\int q\phi$ in place of $q(\phi)$ when $q$ is any distribution.}
However, not all distributions are of that type, the most famous example being $\delta_{x_0}:\phi\mapsto \phi(x_0)$ (assuming $x_0\in\iOmega$).
Both $d_f$ and $\delta_{x_0}$ are distributions of order $0$.

If $q(\phi)=0$ for all $\phi$ whose support is contained in the open set $U\subset\iOmega$ we say that the distribution $q$ vanishes on $U$.
The complement in $\iOmega$ of the largest open set on which $q$ vanishes is called its {\em support}\index{distribution!support of a} and is denoted by $\supp q$.
For example, $\supp \delta_{x_0}=\{x_0\}$.

The set $\mc D'(\iOmega)$ becomes a linear space upon defining $\alpha q_1+\beta q_2$ by $(\alpha q_1+\beta q_2)(\phi)=\alpha q_1(\phi)+\beta q_2(\phi)$ whenever $q_1, q_2\in\mc D'(\iOmega)$ and $\alpha,\beta\in\bb C$.
Similarly, $\mc D^{\prime k}(\iOmega)$ is a linear space for any non-negative integer $k$.

Next note that $\phi\mapsto (-1)^k q(\phi^{(k)})$ is a distribution if $q$ is and if $k$ is a non-negative integer.
This distribution is called the $k$-th {\em derivative}\index{derivative of a distribution} of $q$ and is denoted by $q^{(k)}$.
Note that, if $f$ is locally absolutely continuous, then $(d_f)'=d_{f'}$.

Distributions also have antiderivatives.
To see this fix $\psi\in \mc D(\iOmega)$ with $\int\psi=1$ so that $\varphi(x)=\int(\phi-\psi\int\phi)\chi_{(a,x)}$ defines a test function $\varphi$ for any $\phi\in\mc D(\iOmega)$.
Now, if $q$ is a distribution, define the linear functional $p:\phi\mapsto -q(\varphi)$.
It is easy to check that $p$ is a distribution.
In fact, if $q$ is of order $k>0$, then $p$ is of order $k-1$.
If $q$ is of order $0$, then so is $p$.
Also, since $\int \phi'=0$, we find that $p'(\phi)=-p(\phi')=q(\phi)$, \ie, $p$ is an antiderivative of $q$.
Two antiderivatives of a distribution differ by only a constant as the following lemma shows.

\begin{lem}[Du Bois-Reymond] \label{dBR}
Suppose the derivative of the distribution $p$ is zero.
Then $p$ is the constant distribution, \ie, there is a constant $C$ such that $p(\phi)=C\int\phi$ for all $\phi\in\mc D(\iOmega)$.
\end{lem}

We may define the {\em conjugate}\index{distribution!conjugate} of a distribution $q$ by
$$\ov{q}(\phi)=\ov{q(\ov{\phi})}$$
since $\ov{\phi}$ is a test function if and only if $\phi$ is.
Note that $\ov{\ov q}=q$.
The distribution $q$ is called {\em real}\index{real distribution}\index{distribution!real} if $q=\ov q$.
Equivalently, $q$ is real if $q(\phi)\in\bb R$ whenever $\phi$ assumes only real values.
Finally, $q$ is called \textit{non-negative}\index{non-negative distribution}\index{distribution!non-negative}, if $q(\phi)\geq0$ whenever $\phi\geq0$.
Every non-negative distribution is of order $0$ (H\"ormander \cite{MR1996773}, Theorem 2.1.7).

\subsection{Lebesgue-Stieltjes measures and distributions of order \texorpdfstring{$0$}{0}}\label{ADM.2}
Suppose $I$ is an interval in $\bb R$.
For a function $Q:I\to\bb C$ the \textit{variation}\index{variation of a function} of $Q$ over $I$ is
$$\Var_Q(I)=\sup\Big\{\sum_{j=1}^n |Q(x_j)-Q(x_{j-1})|: x_j\in I, x_0<x_1< ... <x_n\Big\}.$$
If $\Var_Q(I)$ is finite, we say that $Q$ is of \textit{bounded variation}\index{bounded variation} on $I$.
If $\Var_Q(K)$ is finite for all compact subintervals $K$ of $I$, then $Q$ is said to be of \textit{locally bounded variation}\index{locally bounded variation}\index{bounded variation!locally} on $I$.
Clearly, every non-decreasing function is of locally bounded variation.
The set of functions, which are of locally bounded variation, forms a vector space denoted by $\BVl(I)$.
The functions of bounded variation constitute a subspace of $\BVl(I)$, which is denoted by $\BV(I)$.
In fact, $\BV(I)$ is a Banach space with norm $\vertiii Q=|Q(c)|+\Var_Q(I)$ when $c$ is a fixed point in $I$.

Given a function $Q\in\BVl(\iOmega)$ we define the corresponding right- and left-continuous functions $Q^+$ and $Q^-$ by setting $Q^+(x)=\lim_{t\downarrow x}Q(t)$ and $Q^-(x)=\lim_{t\uparrow x}Q(t)$.
We also define $Q^\#(x)=(Q^+(x)+Q^-(x))/2$ which we call the balanced representative of $Q$.
Correspondingly we introduce the spaces $\BVl^+(\iOmega)$, $\BVl^-(\iOmega)$, and $\BVl^\#(\iOmega)$ which collect, respectively, all right-continuous, left-continuous, and balanced functions which are of locally bounded variation on $(a,b)$.

A function $Q\in\BVl(\iOmega)$ generates a finite complex measure $dQ$, called a \textit{Lebes\-gue-Stieltjes measure}\index{Lebesgue-Stieltjes measure}, on any compact subinterval $K$ of $(a,b)$.
Any Lebesgue-Stieltjes measures on $K$ is defined at least on the collection of Borel-measur\-able sets contained in $K$.
In particular, $dQ([x,y])=Q^+(y)-Q^-(x)$ whenever $[x,y]\subset(a,b)$.
Unless $Q$ is either of bounded variation or non-decreasing $dQ$ cannot be extended to a measure on $(a,b)$.
However, the \textit{variation function}\index{variation function} $V_Q$, defined by
$$V_Q(x)=
\begin{cases}
\Var_Q([c,x])&\text{if $x>c$,}\\
0&\text{if $x=c$,}\\
-\Var_Q([x,c])&\text{if $x<c$,}
\end{cases}$$
is non-decreasing and hence generates a positive measure $dV_Q$ on $(a,b)$.
In fact, $dV_{Q^-}$ is the total variation measure of $dQ$ when both are restricted to a compact set.
To simplify notation we put $\hat Q=V_{Q^-}$ in the following.
Since $dQ$ is absolutely continuous with respect to $d\hat Q$, we may define the Radon-Nikodym derivative $h=dQ/d\hat Q$ of $dQ$ with respect to $d\hat Q$.
Note that $h$ has absolute value $1$.
If $f\in L^1(d\hat Q)$ it is customary to write $\int f\ dQ$ for $\int fh\ d\hat Q$.
In particular, if $E\subset K$, then $dQ(E)=\int\chi_E dQ=\int\chi_Eh d\hat Q$.

Next, given a $Q\in\BVl(\iOmega)$, the assignment $\phi\mapsto \int\phi\ dQ$ for $\phi\in\mc D(\iOmega)$ is a distribution of order $0$.
In fact, every distribution of order $0$ is of this nature by the following variant of Riesz's representation theorem.
\begin{thm}
If $q$ is a distribution of order $0$, then there is a function $Q\in\BVl(\iOmega)$ such that $q(\phi)=\int\phi q=\int\phi dQ$.
\end{thm}

The theorem entails that distributions of order $0$ on $(a,b)$ are in one-to-one correspondence with set functions which locally are measures on $(a,b)$ and we will colloquially use the expressions ``measure'' and ``distributions of order $0$'' interchangeably.
Consequently, if $w$ is a non-negative distribution of order $0$, we will use the notation $\mc L^1(w)$ (or $\mc L^1_{\rm loc}(w)$) for the set of functions $f$ which are (locally) integrable with respect to the measure corresponding to $w$.

For functions of locally bounded variation we have the following integration by parts formula (see Hewitt and Stromberg \cite{MR0188387}, Theorem 21.67 and Remark 21.68):
\begin{lem} \label{ibyp}
If $F,G\in \BVl(\iOmega)$, then
$$\int_{[x_1,x_2)} (F^+ dG + G^- dF)=(FG)^-(x_2)-(FG)^-(x_1)$$
whenever $[x_1,x_2]\subset(a,b)$.
\end{lem}
Since $\int_{\{x\}}(F^+dG + G^-dF) = (FG)^+(x)-(FG)^-(x)$ it is easy to extend the integration by parts formula to other kinds of intervals.
This result may be rephrased as a product rule for functions of locally bounded variation:
\begin{equation}\label{170410.3}
(FG)'=F^+G'+F'G^-.
\end{equation}

Lemma \ref{ibyp} shows that the map $\phi\mapsto \int\phi Q$ is an antiderivative of $\phi\mapsto \int\phi dQ$ when $Q\in\BVl(\iOmega)$.
This allows us to use the following notation
$$\int f Q'= \int f dQ$$
whenever $f$ is integrable with respect to $d\hat Q$.

We close this section with the following observation.
Suppose $q$ is a distribution of order zero with antiderivative $Q$ and $f\in\mc L^1_{\rm loc}(d\hat Q)$.
Then
\begin{equation}\label{170414.1}
fq=qf:\phi\mapsto \int \phi f\ dQ
\end{equation}
is also a distribution of order zero.
Moreover, $\ov{qf}=\ov{q}\ov{f}$.

\subsection{Matrix- and vector-valued distributions}\label{ADM.3}
If $\mc K$ is a set of numbers, functions, or operators we denote the set of $m\times n$-matrices whose entries are elements of $\mc K$ by $\mc K^{m\times n}$.
Here $m$ and $n$ denote the number of rows and columns, respectively. If $n=1$ we may write $\mc K^m$ instead of $\mc K^{m\times1}$.
In particular, $\bb C^m$ is a space of columns of $m$ complex numbers.
If $x\in\bb C^m$, then $x^*$ denotes the row whose entries are the complex conjugates of the entries of $x$.
Of course, $\bb C^m$ is a Hilbert space under the scalar product $x^*y$ (which is linear in the second argument).
We denote the elements of the canonical basis in $\bb C^m$ by $\cbv_k$, $k=1, ...,m$, \ie, the $j$-th component of $\cbv_k$ equals $1$ if $j=k$ and $0$ otherwise.

For a matrix-valued function $Q$ we define $\Var_Q$ and $V_Q$ analogously to the case of scalar-valued functions except that we use the $|\cdot|_1$-norm instead of the absolute value.
It is easy to verify that
\begin{equation}\label{170923.1}
\Var_Q(I)=\sum_{j=1}^{m}\sum_{k=1}^{n}\Var_{Q_{j,k}}(I).
\end{equation}
The definitions of bounded and locally bounded variation extend then immediately to the matrix-valued case.
In particular, $\BV(\iOmega)^{m\times n}$ and $\BV^\#(\iOmega)^{m\times n}$ are Banach spaces with norm $\vertiii Q=|Q(c)|_1+\Var_Q(a,b)$ when $c$ is a fixed point in $(a,b)$.

For $q\in\mc D^{\prime0}((a,b))^{m\times n}$ with associated antiderivative $Q$ we define $\Delta_q(x)=Q^+(x)-Q^-(x)$.
If $I$ is a subinterval of $(a,b)$ we have
\begin{equation}\label{180320.1}
\sup_{x\in I}\vert\Delta_q(x)\vert_1\leq\sum_{x\in I}\vert\Delta_q(x)\vert_1\leq\Var_{Q}(I)=\int_I dV_Q
\end{equation}
assuming that $Q$ is left-continuous.

Recall that a matrix $M$ in $\bb C^{n\times n}$ is called non-negative if $a^*Ma\geq0$ for all $a\in\bb C^n$.
Similarly, a $\bb C^{n\times n}$-valued function $W$ defined on some interval is called non-decreasing, if $a^*Wa$ is non-decreasing whenever $a\in\bb C^n$.

If $w\in\mc D^{\prime0}(\iOmega)^{m\times n}$ we define $w^*\in\mc D^{\prime0}(\iOmega)^{n\times m}$ by setting $(w^*)_{k,\ell}=\ov{w_{\ell,k}}$ extending the usual definition of adjoints of matrices of numbers.
A distribution $w\in\mc D^{\prime0}(\iOmega)^{n\times n}$ is called Hermitian if $w^*=w$.
It is called non-negative, if the distribution $a^*wa$ is non-negative for any $a\in\bb C^n$.
Of course, this definition agrees with the previous one, if $n=1$.
Note that a non-negative distribution must be Hermitian and that its diagonal elements and thus its trace must each be non-negative distributions, too.

Now assume that $w$ is a non-negative distribution, \ie, it has a non-decreasing antiderivative $W$, and let $S=\tr W$, a non-decreasing scalar function, be an antiderivative of $\tr w$.
It follows that $S\id-W$ is also non-decreasing
%If $0\leq A$, then the eigenvalues of $A$ and hence $\tr A$ are all non-negative.
%Also, if $x$ is an eigenvector of $A$ associated with $\rho$, then $x$ is also an eigenvector of $(\tr A)-A$ associated with $(\tr A)-\rho\geq0$.
and this, in turn implies that the measures generated by the $W_{j,k}$ on a compact interval $K\subset(a,b)$ are absolutely continuous with respect to $dS$.
Thus, by the Radon-Nikodym theorem, there is a locally $dS$-integrable matrix $\RN$ such that $w=\RN dS$.
For each $a\in\bb C^n$ we have $0\leq a^*\RN a\leq \tr \RN=1$ outside a set of $dS$-measure $0$.
Since vectors with rational components are dense in $\bb C^n$ we find that $0\leq \RN\leq \tr \RN=1$ pointwise almost everywhere with respect to $dS$.

Next we introduce the vector space $\mc L^1_\loc(w)$ of all $\bb C^n$-valued functions $f$ on $\iOmega$ such that each component of $\RN f$ is locally integrable with respect to $dS$.
Similarly, $\mc L^2(w)$ is the vector space of those functions $f$ for which $\int f^*\RN f dS<\infty$.
Note that $\mc L^2(w)$ is contained in $\mc L^1_\loc(w)$.
The assignment $\<f,g\>=\int f^*\RN g dS$ defines a semi-scalar product and thus a semi-norm $\|f\|=\<f,f\>^{1/2}$.
As usual, identifying elements $\mc L^2(w)$ whose differences have norm $0$, one arrives, as first shown by Kac~\cite{MR0080280}, at a Hilbert space which will be denoted by $L^2(w)$.

Finally, following \eqref{170414.1}, we define the distribution $wf\in \mc D^{\prime0}(\iOmega)^{n}$ componentwise by setting $(wf)(\phi)=\int (\RN f) \phi dS$ whenever $f\in \mc L^1_\loc(w)$.

\subsection{Dependence on a parameter}\label{ADM.4}
Suppose $\Omega$ is a subset of $\bb C$ and $r$ a function from $\Omega$ to $\mc D^{\prime0}((a,b))^{n\times m}$ with associated left-continuous antiderivative $R(.,\lambda)$.
If $r_0\in\mc D^{\prime0}((a,b))^{n\times m}$ with left-continuous antiderivative $R_0$, we say $r(\lambda)$ {\em converges} to $r_0$ in {\em variation}\index{convergence in variation} as $\lambda$ tends to $\lambda_0$, a limit point of $\Omega$, if $\lim_{\lambda\rightarrow\lambda_0}\Var_{R(.,\lambda)-R_0}([c,d])=0$ whenever $[c,d]\subset(a,b)$.
The limit $r_0$ is then uniquely determined.
If $\Omega$ is an open subset of $\bb C$, we say that $r$ is \textit{continuous}\index{continuous} in $\Omega$ if $r(\lambda)$ converges to $r(\lambda_0)$ in variation for each $\lambda_0\in\Omega$.
Moreover, $r$ is \textit{analytic}\index{analytic} in $\Omega$, if for each $\lambda_0\in\Omega$ there exists $\dot r(\lambda_0)\in\mathcal{D}^{\prime0}((a,b))^{n\times m}$, called the derivative of $r$ at $\lambda_0$,
such that $(r(\lambda)-r(\lambda_0))/(\lambda-\lambda_0)$ converges to $\dot r(\lambda_0)$ in variation as $\lambda$ tends to $\lambda_0$.
Obviously, if $r$ is analytic in $\Omega$, then $r$ is continuous in $\Omega$.
Furthermore, it follows then from \eqref{180320.1} that $\la \mapsto\Delta_{r(\la)}(x)$ is analytic in $\Omega$ for each $x\in(a,b)$ and its derivative at $\la_0$ is given by $\Delta_{\dot{r}(\la_0)}(x)$.

\begin{thm}\label{T:A.4}
Let $\Omega$ be an open subset of $\bb C$ and let $A\in\bb C^{n\times n}$ be an invertible matrix.
Let $r:\Omega\to\mc D^{\prime0}((a,b))^{n\times n}$ be continuous in $\Omega$.
If $[s,t]\subset(a,b)$ and $K$ is a compact subset of $\Omega$, then the following statements hold:
\begin{itemize}
\item[(a)] The matrix $A+\Delta_{r(\la)}(x)$ is bounded in $[s,t]\times K$.
\item[(b)] If $\det(A+\Delta_{r(\la)}(x))\neq0$ for all $(x,\la)\in[s,t]\times K$, then $\det(A+\Delta_{r(\la)}(x))$ is bounded away from zero in $[s,t]\times K$.
Therefore $(A+\Delta_{r(\la)}(x))^{-1}$ is bounded in $[s,t]\times K$.
\item[(c)] If $r(\la)=\la w-q$ where $q,w\in\mc D^{\prime0}((a,b))^{n\times n}$ and $\det(A-\Delta_q(x))\neq0$ for all $x\in[s,t]$, then the number of points $(x,\la)\in[s,t]\times K$ such that $\det(A+\la \Delta_{w}(x)-\Delta_q(x))=0$ is finite.
\end{itemize}
\end{thm}

\begin{proof}
(a) follows from \eqref{180320.1} and the continuity of $\lambda\mapsto\Var_{R(.,\lambda)}([s,t])$ on $K$.

For (b) assume to the contrary that $j\mapsto(x_j,\la_j)$ is a sequence in $[s,t]\times K$ such that $\lim_{j\rightarrow\infty}\det(A+\Delta_{r(\la_j)}(x_j))=0$.
We may assume that $(x_j,\la_j)$ converges to $(x,\la)\in[s,t]\times K$.
If $\{x_j:j\in\bb N\}$ is finite, this implies that $\lim_{j\rightarrow\infty}\det(A+\Delta_{r(\la_j)}(x))=0$.
But $\lim_{j\rightarrow\infty}\vert\Delta_{r(\la_j)}(x)-\Delta_{r(\la)}(x)\vert_1=0$ so $\det(A+\Delta_{r(\la)}(x))=0$ which is a contradiction.
On the other hand, if $\{x_j:j\in\bb N\}$ is infinite, then we may assume that $j\mapsto x_j$ is a sequence of distinct points.
It follows from $\vert\Delta_{r(\la_j)}(x_j)\vert_1\leq\vert\Delta_{r(\la_j)-r(\la)}(x_j)\vert_1+\vert\Delta_{r(\la)}(x_j)\vert_1$ and \eqref{180320.1} that
$\lim_{j\to\infty}\vert\Delta_{r(\la_j)}(x_j)\vert_1=0$ which implies that $\det(A)=0$.
This contradiction finishes the proof of the first part of (b).
The last statement in (b) follows since the inverse of a matrix may be written in terms of its cofactors and its determinant.

To prove (c) note that $\Delta_{r(\la)}(x)=\la\Delta_w(x)-\Delta_q(x)$.
Assume to the contrary that $j\mapsto(x_j,\la_j)$ is a sequence of distinct points in $[s,t]\times K$ such that $\det(A+\Delta_{r(\la_j)}(x_j))=0$.
Note that the invertibility of $A-\Delta_q(x)$ for all $x\in[s,t]$ implies that the set $\{x_j:j\in\bb N\}$ is infinite.
Therefore we may assume that $j\to x_j$ is a sequence of distinct points and that $j\to\la_j$ converges to, say, $\la$.
But $\vert\Delta_{r(\la_j)}(x_j)\vert_1\leq\vert\Delta_{r(\la_j)-r(\la)}(x_j)\vert_1+\vert\Delta_{r(\la)}(x_j)\vert_1$ and the right-hand side tends to zero as $j$ tends to infinity so $\det(A)=0$ and this contradiction completes the proof of (c).
\end{proof}

\section{Linear relations}\label{ALR}
We present here the basic facts about linear relations and their spectral theory.
It appears that linear relations were introduced by Arens \cite{MR0123188} while Orcutt \cite{Orcutt} was the first to use them to investigate a system of differential equations.
Our presentation in Sections \ref{ALR.1} -- \ref{ALR.3} follows closely Bennewitz's paper \cite{MR0454398} but we added Theorem \ref{T:B.6}.

\subsection{Basic definitions}\label{ALR.1}
Let $\mc H$ and $\mc H_k$ (for various $k$) be Hilbert spaces with scalar products $\<\cdot,\cdot\>$ and $\<\cdot,\cdot\>_k$, respectively.
We consider $\mc H_1\exsum\mc H_2$ to be the external direct sum of $\mc H_1$ and $\mc H_2$, \ie, a Hilbert space with scalar product $\<(u,f),(v,g)\>=\<u,v\>+\<f,g\>$.
A (closed) \textit{linear relation}\index{linear relation}\index{relation} in $\mc H_1\exsum\mc H_2$ is a (closed) linear subspace of $\mc H_1\exsum\mc H_2$.
The {\em domain}\index{domain of a relation} and the {\em range}\index{range of a relation} of a linear relation $S$ in $\mc H_1\exsum\mc H_2$ are the sets $\dom(S)=\{u\in\mc H_1:\exists f: (u,f)\in S\}$ and $\ran(S)=\{f\in\mc H_2:\exists u: (u,f)\in S\}$, respectively.
The set $\ker(S)=\{u\in\mc H_1: (u,0)\in S\}$ is called the \textit{kernel}\index{kernel} of $S$.
Clearly, $\dom(S)$ and $\ker(S)$ are subspaces of $\mc H_1$ and $\operatorname{ran}(S)$ is a subspace of $\mc H_2$.
If $T$ is also a linear relation and $S\subset T$, then $T$ is called an \textit{extension}\index{extension} of $S$ and $S$ is called a \textit{restriction}\index{restriction} of $T$.
A linear relation $S\subset\mc H_1\exsum\mc H_2$ is called \textit{densely defined}\index{densely defined} if $\dom(S)$ is a dense subset of $\mc H_1$.
A linear relation $S$ is called a \textit{linear operator}\index{linear operator}\index{operator}, if $(u,f)\in S$ and $(u,g)\in S$ imply $f=g$.
A linear operator $S\subset \mc H_1\exsum\mc H_2$ is called \textit{bounded}\index{bounded linear operator}, if there is a number $C\geq0$ such that $\|f\|_2\leq C \|u\|_1$ whenever $(u,f)\in S$.
Note that with these notations we do not distinguish between a relation (or an operator) and its graph.
Nevertheless we will sometimes write $S:\dom(S)\to \mc H_2$ instead of $S\subset \mc H_1\exsum\mc H_2$ and $Sx=f$ instead of $(x,f)\in S$ when $S$ is an operator.
If $\mc H=\mc H_1=\mc H_2$ we denote the {\em identity operator}\index{identity operator} $\{(u,u): u\in \mc H\}$ by $\id$.

Next we define an addition and a scalar multiplication of linear relations in $\mc H_1\exsum\mc H_2$.
Suppose $S$ and $T$ are two such relations and $\alpha$ is a complex number.
Then we set
$$S+T=\{(u,f+g): (u,f)\in S, (u,g)\in T\}$$
and
$$\alpha S=\{(u,\alpha f): (u,f)\in S\}.$$
In particular, the domain of $S+T$ is the intersection of the domains of $S$ and $T$.
We emphasize that this notion of addition must not be confused with a (direct) sum of subspaces.
The operation of addition is associative and commutative and the {\em zero operator}\index{zero operator} $\{(u,0):u\in\mc H_1\}$ is the additive identity element.
However, not every relation has an additive inverse.
Indeed to have an additive inverse it is necessary (and sufficient) for a relation to be an everywhere defined operator.

We also define composite relations
$$S\circ T=\{(u,w)\in\mc H_1\exsum\mc H_3: \exists v\in\mc H_2: (u,v)\in T, (v,w)\in S\},$$
if $T$ and $S$ are linear relations in $\mc H_1\exsum\mc H_2$  and $\mc H_2\exsum\mc H_3$, respectively.
We will mostly abbreviate $S\circ T$ by $ST$.

Each linear relation has an inverse and an adjoint.
The {\em inverse}\index{inverse of a relation} of $S\subset\mc H_1\exsum\mc H_2$ is
$$S^{-1}=\{(f,u)\in\mc H_2\exsum\mc H_1: (u,f)\in S\}.$$
The \textit{adjoint}\index{adjoint} of $S\subset\mc H_1\exsum\mc H_2$ is
$$S^*=\{(v,g)\in\mc H_2\exsum\mc H_1: \forall (u,f)\in S: \<g,u\>_1=\<v,f\>_2\}.$$
One checks easily that $S^{-1}$ and $S^*$ are linear relations themselves.

Suppose $S$ and $T$ are linear relations in $\mc H_1\exsum\mc H_2$.
Then the following statements hold.
 \begin{enumerate}
  \item If $S\subset T$, then $T^*\subset S^*$.
  \item $S^*$ is a closed linear relation and $S^{**}=\ov S$, the closure of $S$.
  \item If $S$ is a linear relation, then $\ker S^*=(\ran S)^\perp$.
 \end{enumerate}

A linear relation $E$ is called {\em symmetric}\index{symmetric}, if $E\subset E^*$ and {\em self-adjoint}\index{self-adjoint} if $E=E^*$.
Note that, if $E$ is symmetric, then so is its closure $\ov E$.

\subsection{Resolvents}\label{ALR.2}
Throughout this section we assume that $E$ is a linear relation in $\mc H\exsum\mc H$.
For $\la\in\bb C$ we define the {\em deficiency spaces}\index{deficiency space} $D_\la=\{(u,\la u)\in E^*\}$.

If $E$ is closed, one defines the set $\rho(E)$ of those complex numbers $\la$ for which $(E-\la)^{-1}$ is a closed linear operator with domain $\mc H$, \ie,
$$\rho(E)=\{\la\in\bb C: \ker(E-\la)=\{0\}, \ran(E-\la)=\mc H\}.$$
$\rho(E)$ is called the {\em resolvent set}\index{resolvent set} of $E$.
The set $\sigma(E)=\bb C\setminus\rho(E)$ is called the \textit{spectrum}\index{spectrum} of $E$.
In particular, an \textit{eigenvalue}\index{eigenvalue} of $E$, \ie, a number $\la$ such that $\ker(E-\la)\neq\{0\}$, is always in $\sigma(E)$.
The operator $(E-\la)^{-1}$, commonly denoted by $R_\la$, is called the {\em resolvent} of $E$ at $\la$.
The closed graph theorem shows that there is a constant $C$ such that $\|v\|\leq C\|u\|$ whenever $\la\in\rho(E)$ and $(u,v)\in R_\la$, \ie, $R_\la$ is a bounded linear operator when $\la\in\rho(E)$.

\begin{thm}\label{T:B.2}
The resolvent set and the resolvent of a closed linear relation $E$ have the following properties.
\begin{enumerate}
  \item The resolvent set $\rho(E)$ is open.
  \item If $\la,\mu\in\rho(E)$ the \textit{resolvent relation}\index{resolvent relation}
        $R_\la-R_\mu=(\la-\mu) R_\la R_\mu$ holds.
  \item If $E$ is self-adjoint, then $\bb C\setminus\bb R\subset \rho(E)$.
  \item If $E$ is self-adjoint, then $R_\la^*=R_\ol$.
\end{enumerate}
\end{thm}

\subsection{Extension theory for symmetric relations}
\begin{thm}\label{vonNeumann1}
If $E$ is a closed symmetric relation in $\mc H\exsum\mc H$ and $\la\in\bb C\setminus\bb R$, then $E^*=E\dot+D_\la\dot+D_{\ol}$.
The space $D_\la\dot+D_{\ol}$ is closed.
If $\lambda=\pm i$ the sum is, in fact, orthogonal, \ie, $E^*=E\oplus D_i\oplus D_{-i}$.
\end{thm}

\begin{cor} \label{C:B.3}
As long as $\la$ remains in either the upper or the lower half plane $\dim D_\la$ is independent of $\la$.
\end{cor}

The dimensions of $D_{\pm i}$, denoted by $n_\pm$, are called {\em deficiency indices} of $E$.
Note that, if $E$ is not closed, its deficiency spaces coincide with those of $\ov E$.

We now want to characterize the symmetric extensions of $E$.
Let $V=D_i\oplus D_{-i}$ and $d=\dim V=n_++n_-$.
We shall use the operator $\mc J:\mc H\exsum\mc H\to\mc H\exsum\mc H: (u,f)\mapsto (f,-u)$.
If we denote the scalar product of the Hilbert space $\mc H\exsum\mc H$ also by $\<\cdot,\cdot\>$ we get $\<g,u\>-\<v,f\>=\<\mc J(v,g),(u,f)\>$ and hence $\mc J(S^*)=S^\perp$ and $S^*=\mc J(S^\perp)$ for any relation $S\subset\mc H\exsum\mc H$.
Also note that $\dom(V)=\ran(V)$ and $\mc J(V)=V$.

First we note the following fact.
If $F$ is a closed symmetric extension of the closed symmetric relation $E$ in $\mathcal{H}\exsum\mathcal{H}$, then
\begin{equation}\label{fL:B.5}
\dim(F\ominus E)=\dim(E^*\ominus F^*)\leq\dim(E^*\ominus F).
\end{equation}

\begin{thm}\label{vonNeumann2}
Let $E$ be a closed symmetric relation in $\mathcal{H}\exsum\mathcal{H}$.
Then $F$ is a closed symmetric extension of $E$ if and only if
$$F=E\oplus\{(u+v,i(u-v)):(u,v)\in L\},$$
where $L\subset (\dom D_i)\exsum(\dom D_{-i})$ is a norm-preserving closed linear operator.
Moreover, $F$ is self-adjoint if and only if $\dom L=\dom D_i$ and $\ran L=\dom D_{-i}$.
\end{thm}

While Theorem \ref{vonNeumann2} is standard, the following characterization is, to our best knowledge, new.
\begin{thm}\label{T:B.6}
Suppose $E$ is a closed symmetric relation in $\mc H\exsum\mc H$ with $d=\dim V<\infty$ and that $m\leq d/2$ is a natural number or $0$.
If $A:E^*\to\bb C^{d-m}$ is a surjective linear operator such that $E\subset \ker A$ and $A\mc JA^*$ has rank $d-2m$ then $\ker A$ is a closed symmetric extension of $E$ for which the dimension of $(\ker A)\ominus E$ is $m$.
Conversely, every proper closed symmetric extension of $E$ is the kernel of such a linear operator $A$.
Finally, $\ker A$ is self-adjoint if and only if $A\mc JA^*=0$ (entailing $m=d/2$).
\end{thm}

\begin{proof}
Suppose that $A\subset E^*\exsum\bb C^{d-m}$ is given with the stated properties.
Since $E$ is closed and $V$ is finite-dimensional, $\ker A$ is closed and $A$ itself is bounded.
Therefore $A^*\subset \bb C^{d-m}\exsum E^*$ is a bounded linear operator defined on all of $\bb C^{d-m}$.
Its range is orthogonal to $\ker A$ and its kernel is orthogonal to $\ran A$.
Thus $A^*$ is injective and $\ran A^*=E^*\cap (\ker A)^\perp$.
This implies $\ran(\mc JA^*)=\mc J(\ran A^*)=(\ker A)^*\cap E^\perp\subset \dom A$.
Hence $B=A\mc J A^*$ is well defined, in fact it is a linear operator from $\bb C^{d-m}$ to itself.
Now note that $S=\mc JA^*(\ker B)$ has dimension $m$, the same as $\ker B$, and that it is a subset of $(\ker A)\cap E^\perp$.
However, $(\ker A)\cap E^\perp=\ker(A|_V)$ also has dimension $m$ since $A|_V$ is surjective.
It follows that $(\ker A)\cap E^\perp=S\subset\ran(\mc JA^*)\subset (\ker A)^*$.
Since $\ker A\subset E^*$ we also have $E\subset (\ker A)^*$ and thus $\ker A\subset (\ker A)^*$, \ie, $\ker A$ is symmetric.

For the converse, suppose that $F$ is a closed symmetric extension of $E$.
Let $D=F\ominus E$, $D'=E^*\ominus F$, and $m=\dim D$.
Since $V=D\oplus D'$, the dimension of $D'$ is $d-m$.
Thus \eqref{fL:B.5} implies that $2m\leq d$.
Pick an orthonormal basis $(v_1,g_1),(v_2,g_2),...,(v_{d-m},g_{d-m})$ of $D'$ and define $A:E^*\to\bb C^{d-m}$ by $A_j(u,f)=\<(v_j,g_j),(u,f)\>$.
Then $A:E^{*}\mapsto\bb C^{d-m}$ is a surjective linear operator whose kernel is $F$ and thus contains $E$.
It follows that $A^*$ is injective and $\ran A^*=D'$.
Again $B=A\mc JA^*$ is well defined and we need to show that $\dim\ker B=m$.
As before we have $\mc JA^*(\ker B)\subset(\ker A)\cap E^\perp=D$ so that $\dim\ker B\leq m$.
Also, if $U$ denotes the $m$-dimensional subspace of $\bb C^{d-m}$ which is mapped by $\mc JA^*$ to $D$, then $U\subset\ker B$ implying $m\leq\dim\ker B$.

Finally, \eqref{fL:B.5} implies that $m=d/2$ if and only if $\ker A=(\ker A)^*$.
\end{proof}

\subsection{Reduction of a self-adjoint relation to a self-adjoint operator}\label{ALR.3}
Suppose $E$ is a self-adjoint linear relation in $\mc H\exsum\mc H$ and define $\mc H_\infty=\{g\in \mc H:(0,g)\in E\}$.
Then $\mc H_\infty$ is a closed subspace of $\mc H_1$.
The following theorem associates an operator, densely defined in $\mc H_0=\mc H_\infty^\perp$, to $E$.

\begin{thm}\label{T:ALR.3.1}
Suppose $E$ is self-adjoint linear relation in $\mc H\exsum\mc H$.
Then the following two statements are true.
  \begin{enumerate}
    \item The domain of $E$ is a dense subset of $\mc H_0$.
    \item $E_0=E\cap(\mc H_0\exsum\mc H_0)$ is a densely defined self-adjoint operator in $\mc H_0\exsum\mc H_0$.
  \end{enumerate}
\end{thm}

$E_0$ is called the {\em operator part}\index{operator part} of $E$.

We will also need the following result about the relationships between the resolvents of $E$ and $E_0$.
Of course, we consider $E_0$ as a relation $\mc H_0\exsum \mc H_0$.
\begin{thm}\label{T:ALR.3.2}
If $E$ is a self-adjoint linear relation in $\mc H\exsum \mc H$, then $\rho(E)=\rho(E_0)$.
Moreover, the resolvent $R_\la$ of $E$ annihilates the space $\mc H_\infty$, \ie, $R_\la f=0$ whenever $f\in\mc H_\infty$.
Finally, the resolvent of $E_0$ is given by $R_\la\cap(\mc H_0\exsum\mc H_0)$.
\end{thm}

\subsection{The spectral theorem for self-adjoint operators}\label{ALR.4}
For easy reference and to fix notation we state here the spectral theorem for self-adjoint operators.
See Rudin \cite{MR1157815} for additional details.

\begin{defn}
Suppose $\mc M$ is a $\sigma$-algebra on $\bb R$ which includes the Borel sets and $\pi$ is a function on $\mc M$ whose values are orthogonal projections in the Hilbert space $\mc H$.
$\pim$ is called a {\em resolution of the identity}\index{resolution of the identity} if it has the following properties.
\begin{enumerate}
  \item $\pim(\emptyset)=0$ and $\pim(\bb R)=\id$.
  \item $\pim(B\cap B')=\pim(B)\pim(B')$ whenever $B, B'\in\mc M$.
  \item If $B, B'\in\mc M$ are disjoint, then $\pim(B\cup B')=\pim(B)+\pim(B')$.
  \item For every $f,g\in\mc H$ the function $t\mapsto \Pim_{f,g}(t)=\<f,\pim((-\infty,t))g\>$ is left-continuous and of bounded variation and hence induces a complex measure defined on $\mc M$.
\end{enumerate}
$\pim(B)$ is called the {\em spectral projection}\index{spectral projection} of $B$.
\end{defn}

\begin{thm}\label{T:B.8}
Suppose $E_0\subset\mc H_0\exsum\mc H_0$ is a densely defined, self-adjoint operator.
Then there exists a unique resolution of the identity $\pim$ such that
$$\<f,E_0g\>=\int t\ d\Pim_{f,g}(t)$$
for any $f\in\mc H_0$ and $g\in\dom E_0$.

Moreover, $\pim$ is concentrated on $\sigma(E_0)$, \ie, $\pim(A)=0$ for $A\subset \rho(E_0)\cap\bb R$, and
$$\<f, R_\la g\>=\int \frac1{t-\la}\ d\Pim_{f,g}(t)$$
for $\la\in\rho(E_0)$.
\end{thm}

If $E$ is a self-adjoint relation in $\mc H\exsum\mc H$ and $E_0\subset\mc H_0\exsum\mc H_0$ is its operator part, we extend the domain of definition of the spectral projections $\pim(B)$ from $\mc H_0$ to $\mc H$ by setting $\pim(B)f=0$ whenever $f\in \mc H_\infty$.
Thus $\pim(\bb R)$ becomes the orthogonal projection from $\mc H$ onto $\mc H_0$.

\bibliographystyle{plain}
%\bibliography{e:/public_html/bibtexbrowser/all}
\def\cprime{$'$}

\end{document}